\documentclass[11pt]{amsart}

\usepackage{mathrsfs}
\usepackage{amsmath, amsthm, amsfonts}
\usepackage[percent]{overpic}
\usepackage{color}
\usepackage{MnSymbol}
\usepackage{fullpage}
\usepackage{bbm}

\usepackage{thmtools}
\usepackage{thm-restate}

\usepackage{hyperref}

\usepackage{cleveref}

\declaretheorem[name=Theorem,numberwithin=section]{thm}

\newtheorem*{thm*}{Theorem}

\newtheorem{prop}[thm]{Proposition}
\newtheorem{lem}[thm]{Lemma}

\theoremstyle{definition}
\newtheorem{defn}[thm]{Definition}
\newtheorem{question}[thm]{Question}
\newtheorem*{ack}{Acknowledgements}
\newtheorem{ex}[thm]{Example}
\theoremstyle{remark}
\newtheorem{rem}[thm]{Remark}

\newcommand{\QQ}{\mathbb{Q}}

\newcommand{\HH}{\mathcal{H}}
\newcommand{\PPP}{\mathcal{P}}

\newcommand{\RR}{\mathbb{R}}
\newcommand{\CC}{\mathbb{C}}

\newcommand{\X}{\mathcal{X}}
\newcommand{\ZZZ}{\mathcal{Z}}

\newcommand{\ord}{\operatorname{ord}}

\newcommand{\ZZ}{\mathbb{Z}}

\newcommand{\PP}{\mathbb{P}}

\newcommand{\Pic}{\text{Pic}}
\newcommand{\Mg}{\mathcal{M}_g}
\newcommand{\Mgn}{\mathcal{M}_{g,n}}

\newcommand{\Mgnb}{\overline{\mathcal{M}}_{g,n}}
\newcommand{\MOnb}{\overline{\mathcal{M}}_{0,n}}

\newcommand{\Hur}{\operatorname{Hur}}

\newcommand{\res}{\text{res}}

\newcommand{\SL}{\mbox{SL}}

\newcommand{\Mbar}[2]{\overline{\mathcal{M}}_{{#1}, {#2}}}

\newtheorem{Thm*}{Theorem*}
\theoremstyle{definition}

\title{Strata of differentials of the second kind, positivity and irreducibility of certain Hurwitz spaces
}
\author{Scott Mullane}
\date{\today}

\AtEndDocument{\bigskip{\footnotesize%
\textsc{Scott Mullane, Institut f\"{u}r Mathematik, Goethe-Universit\"{a}t Frankfurt, Robert-Mayer-Str. 6-8, 60325 Frankfurt am Main, Germany}\par
\textit{E-mail address}: \texttt{mullane@math.uni-frankfurt.de} \par
}}

\begin{document}
\thispagestyle{empty}

\maketitle

\begin{abstract}
We consider two applications of the strata of differentials of the second kind (all residues equal to zero) with fixed multiplicities of zeros and poles: \\
\\
\textbf{Positivity: }In genus $g=0$ we show any associated divisorial projection to $\overline{\mathcal{M}}_{0,n}$ is $F$-nef and hence conjectured to be nef. We compute the class for all genus when the divisorial projection forgets only simple zeros and show in these cases the genus $g=0$ projections are indeed nef. 
\\
\\
\textbf{Hurwitz spaces: } 
We show the Hurwitz spaces of degree $d$, genus $g$ covers of $\PP^1$ with pure branching (one ramified point over the branch point) at all but possibly one branch point are irreducible if there are at least $3g+d-1$ simple branch points or $d-3$ simple branch points when $g=0$.

\end{abstract}

\setcounter{tocdepth}{1}

%\tableofcontents

%%%%%%%%%%%%%%%%%%%%%%%%%%%%%
%%%%%%%%%%%%%%%%%%%%%%%%%%%%%
%%%%%%%%%%%%%%%%%%%%%%%%%%%%%
\section{Introduction}

The moduli space of abelian differentials $\HH(\kappa)$ consists of pairs $(C,\omega)$ where $\omega$ is a holomorphic or meromorphic differential on a smooth genus $g$ curve $C$ and the multiplicity of the zeros and poles of $\omega$ is fixed of type $\kappa$, an integer partition of $2g-2$. 
 We define the \emph{stratum of zero residue abelian differentials with signature $\kappa=(k_1,...,k_m)$} as
\begin{equation*}
\HH_Z(\kappa):=\{[C,\omega]\in \HH(\kappa)   \hspace{0.15cm}| \hspace{0.15cm}\res_{p}(\omega)=0\text{ for all $p$ non-simple poles of $\omega$}\}.
\end{equation*}
Further, as our interest is in how these strata inform the birational geometry of the Deligne-Mumford compactification of the moduli space of genus $g$ curves with $n$ ordered points, $\Mgnb$, we define the \emph{stratum of zero residue canonical divisors with signature $\kappa=(k_1,...,k_m)$} as
\begin{equation*}
\ZZZ(\kappa):=\{[C,p_1,...,p_m]\in {\mathcal{M}}_{g,m}   \hspace{0.15cm}| \hspace{0.15cm}k_1p_1+...+k_mp_m\sim(\omega)\sim K_C \text{ and $\res_{p_i}(\omega)=0$ for $k_i\leq -2$}\}.
\end{equation*}
Note that the marked points of $\Mgnb$ are ordered, while $\HH(\kappa)$ records only the differential and the zeros and poles of the differential are unordered.

Integrating a holomorphic $\omega$ or \emph{differential of the first kind} by the periods of $C$ gives a cohomology class in $H^{1,0}(C)$ of the Hodge filtration of $H^1(C)$. Similarly, a meromorphic $\omega$ or \emph{differential of the third kind} gives a class in $H^1(\mathcal{U})$ where $\mathcal{U}$ is the curve $C$ punctured at the poles of $\omega$.
% H1,0, H1 U, known as first and second kind - third kind
In the space between lie the \emph{differentials of the second kind} which are meromorphic differentials with zero residues at every pole, hence specifying a class in $H^1(C)$. The differentials of this kind that specify the zero class are exact differentials.
% Dualising sheaf on cuspidal curves, hence related to projective geometry, 
Differentials of the second kind arise naturally in many areas of algebraic geometry, for example, as sections of the dualising sheaf on cuspidal curves and hence from many constructions in projective geometry. 
In local charts via the Veech zippered rectangle construction (see Section~\ref{charts}) zero residue conditions are cut out by real linear equations in period coordinates. In the holomorphic case, loci with this property are known as affine invariant submanifolds and hold great dynamical importance~\cite{McMullen}\cite{EskinMirzakhani}\cite{EskinMirzakhaniMohammadi}\cite{Filip}\cite{McMW}. Further, restricting within the strata of differentials of the second kind via linear equations in period coordinates to the differentials $[C,\omega]$ such that $\int_\alpha\omega=0$ for all absolute periods $\alpha\in H_1(C)$, we obtain exact differentials, that is, differentials that arise as the result of pulling back $dz$ on the projective line under a finite cover. Hence the strata of differentials of the second kind form an intermediary between the modern study of the strata of abelian differentials and the more classical questions arising from the study of Hurwitz spaces.
 
The strata of canonical divisors and projections of these strata under morphisms forgetting marked points has provided many extremal and interesting cycles in $\Mgnb$ for $g\geq 1$~\cite{ChenCoskun}\cite{Mullane3}\cite{Mullane4}. Specifying a signature does not impose a condition on the position of points on a rational curve, though there is evidence that placing further conditions on the residues at poles provides interesting cycles in $\overline{\mathcal{M}}_{0,n}$. For example, many of the known extremal rays of the effective cone of $\overline{\mathcal{M}}_{0,n}$ arise in this way. The Keel Vermiere divisor~\cite{KV} for $n=6$, some of the Hypertree divisors~\cite{CastravetTevelevHypertree} (generalising the construction to $n\geq 7$, conjectured to give a complete description of the effective cone) and the divisors of Opie~\cite{Opie}  (which disproved the conjecture) for $n\geq 7$ can be viewed as a residue condition on a fixed signature in this way.

A major open question on the birational geometry of $\Mgnb$ is the $F$-conjecture\footnote{Named in honour of Faber and Fulton}. 
\vspace{0.4cm}
\\
\textbf{The $F$-conjecture. }The effective cone of curves in $\overline{\mathcal{M}}_{g,n}$ is generated by $F$-curves.
\vspace{0.4cm}
\\
Where an $F$-curve is defined as an irreducible component of the 1-dimensional locus of stable curves containing at least $3g+n-4$ nodes in $\Mgnb$. Dually, the conjecture can be stated as: $F$-nef divisors are nef, where a divisor is $F$-nef if it has non-negative intersection with all $F$-curves. Gibney, Keel and Morrison~\cite{GKM} showed that the higher genus case follows from the $g=0$ case where the conjecture remains open for $n\geq 8$. In the first application considered in this paper we investigate the positivity properties of the strata of zero residue differentials in general genus with a view to the $F$-conjecture in genus $g=0$.

%%%%%%%% Intro for Hurwitz spaces %%%%%%%%%%%%%
Consider the Hurwitz space of type $\underline{\eta}$ for 
$$\underline{\eta}:=[\eta_1,\dots,\eta_s]$$
an unordered set of integer partitions of $d$, that is, the space of all degree $d$ covers with a branch profile given by $\underline{\eta}$ above the $s$ branch points,
$$   \Hur(\underline{\eta}):=\{\hspace{0.2cm}f:C\longrightarrow \PP^1\hspace{0.2cm}\big|\hspace{0.2cm}\text{The ramification of $f$ is of type $\underline{\eta}$}\}.$$                                
We further define $\eta_i$ to be a \emph{pure} partition if it contains only one entry not equal to one. If $\underline{\eta}$ has exactly one\footnote{With a little more care we actually deal with at most one non-pure branch point.} non-pure branch point, say $\eta_1$, then pulling back a differential on $\PP^1$ with a unique double pole at the branch point associated to $\eta_1$ we obtain a differential of the second kind $\omega$ on $C$ with poles at the points in the special fibre and zeros at the non-trivially ramified points away from the special fibre. Further, for any absolute cycle $\gamma\in H^1(C)$ we have $\int_\gamma\omega=0$.

In the second application considered in this paper we use this identification to produce results on the irreducibility of certain Hurwitz spaces. Though this identification relates to exact differentials, a subvariety of the strata of differentials of the second kind, we include this application as the hard earned technical results in the strategy are all inducted from the genus $g=0$ case where exact differentials and differentials of the second kind coincide.

%%%%%%% Compactification of zero residue conditions %%%%%%%%%%%%%%
\subsection{Compactification} 
The question of how a condition on the residues of a meromorphic differential degenerates to stable nodal curves is answered in Section~\ref{degen}. Through the use of flat geometry and the techniques of~\cite{BCGGM1} we provide the following theorem.

\begin{restatable}{thm}{ZeroResidueGRC}(\textbf{Zero Residue Global Residue Condition})\label{ZRGRC}
 A twisted canonical divisor of type $\kappa$ is the limit of twisted canonical divisors on smooth curves with zero residues at specified poles (denoted \emph{zeroed poles}) if and only if there exists a twisted differential $\{\eta_j\}$ on $C$ such that 
\begin{enumerate} 
\item
If $q$ is a node of $C$ and $q\in C_i\cap C_j$ such that $\ord_q(D_i)=\ord_q(D_j)=-1$ then $\res_q(\eta_i)+\res_q(\eta_j)=0$; 
\item
the residues in each $\eta_j$ are zero at the specified zeroed poles; and
\item
there exists a full order on the dual graph $\Gamma$, written as a level graph $\overline{\Gamma}$, agreeing with the order of $\sim$ and $\succ$, such that for any level $L$ and any connected component $Y$ of  $\overline{\Gamma}_{>L}$ that does not contain a prescribed non-zeroed pole we have
\begin{equation*}
\sum_{\begin{array}{cc}\text{level}(q)=L, \\q\in C_i\subset Y\end{array}}\res_{q}(\eta_i)=0
\end{equation*} 
\end{enumerate}
\end{restatable}

%%%%%%%%%%Definition of divisor classes %%%%%%%%%% 
\subsection{Divisor class computation}
Let $\kappa=(k_1,...,k_r)$ with $m$ entries $k_i\leq-2$ and $\sum k_i=2g-2$. The above theorem gives a full compactification $\overline{\ZZZ}(\kappa)$ of the strata $\ZZZ(\kappa)$. In general, this compactification is not smooth. Proposition~\ref{existence} states when $\ZZZ(\kappa)$ is empty and hence Proposition~\ref{affinemanifold} shows that outside of some genus $g=0$ exceptions, $\overline{\ZZZ}(\kappa)$ forms a codimension $g+m-1$ or $g+m$ subvariety of $\Mbar{g}{r}$ for signature $\kappa $ with no simple poles or at least two simple poles respectively. Pushing forward this class under the morphism forgetting the last $r-n$ points
$$\Mbar{g}{r}\longrightarrow \Mbar{g}{n}$$
we obtain an effective divisor when $n=r-g-m+2$ or $n=r-g-m+1$ for signature $\kappa $ with no simple poles or at least two simple poles respectively. We use the notation $Z_\kappa^n$ to denote both the divisor and the class\footnote{Our class expressions are given modulo the labelling of the unmarked points. See Definition~\ref{divdefn} for an explicit description.} of the divisor in $\Pic(\Mgnb)\otimes \QQ$. In Section~\ref{thm:ZRdivI} we compute the class of this divisor in all cases where the forgotten points are simple zeros.

%%%%% Divisor class table %%%%%%%%%
\vspace{0.4cm}
\begin{center}
\begin{tabular}{ |p{2.5cm}|p{9.5cm}|p{3cm}|  }
 \hline
 \multicolumn{3}{|c|}{Classes of effective divisors in $\Mgnb$ from zero residue strata of abelian differentials $\begin{array}{c}\text{   }\\ \text{   } \end{array}$} \\
 \hline
 $g$ and $n$ & Signature $\kappa$ & Reference  \\
 \hline\hline
 $g\geq 0$, $n\geq 3$& $\kappa=(d_1,\dots,d_n,1^{g+m-2})$ for $d_i\ne-1$ and $d_i\leq-2$ for $m\geq2$ entries and $\sum d_i=g-m$& Theorem~\ref{thm:ZRdivI} $\begin{array}{c}\text{   }\\ \text{   } \end{array}$\\
 \hline
 $g\geq 0$, $n\geq k+2$ &$\kappa=(d_1,\dots,d_n,1^{g+m-2})$ for $d_i=-1$ for $k\geq 2$ entries and $d_i\leq-2$ for $m\geq 1$ entries and $\sum d_i=g-m-1$  & Theorem~\ref{thm:ZRdivII} $\begin{array}{c}\text{   }\\ \text{   } \end{array}$\\
 \hline
 \hline
\end{tabular}
\end{center}
\vspace{0.4cm}

%%%%%%% Divisor class example %%%%%%%%%%
For example, the class of the divisor $[Z^{3}_{-2,-2,4,1^{2}}]$ in $\overline{\mathcal{M}}_{2,3}$ is given by the formula in Theorem~\ref{thm:ZRdivI} as
\begin{equation*}
-\lambda+\psi_1+\psi_2+10\psi_3-\delta_{1:\emptyset}-3\left(\delta_{1:\{1\}}+\delta_{1:\{2\}}+\delta_{0:\{1,2\}}+\delta_{0:\{1,2,3\}}  \right)-6\left(\delta_{0:\{2,3\}} +\delta_{0:\{1,3\}}+\delta_{1:\{3\}} \right).
\end{equation*}
%%%%%%%%%%%%%%%%%%%%%%%%%%%%%

In the genus $g=0$ case the classes of Theorem~\ref{thm:ZRdivI} correspond to classes of divisors introduced by Fedorchuk~\cite{Fed} which leads to a new geometric description of these divisors in terms of Hurwitz spaces that we present as Corollary~\ref{DivCor}.

%%%%%%%Positivity %%%%%%%%%%%%%%%%%
\subsection{Positivity}
%%% Restricting to curves of compact type %%%%%%%%
The positivity properties of the divisor classes are due to a number of results relating to the restriction of strata of differentials of the second kind to the boundary of the moduli space. Restricting to stable curves of compact type and strata of differentials of the second kind we obtain the following result.
\begin{restatable}{thm}{Compacttype}\label{CT}
Fix $\kappa$ such that $|\kappa|=n$ and $-1\nin\kappa$. A stable pointed curve of compact type  $[X,p_1,\dots,p_n]\in \Mgnb$ with irreducible components $\{X_i\}$ is contained in the closure $\overline{\ZZZ}(\kappa)$ if and only if there exists an associated twisted differential $(X,\{\eta_{X_i}\})$ of type $\kappa$ such that each $\eta_{X_i}$ is of the second kind.
\end{restatable}
%%% Restricting to g=0 we have the codim result %%%%%
Restricting to the genus $g=0$ case where all curves are of compact type, we obtain the following result regarding restriction to boundary strata.
\begin{restatable}{lem}{Restrictionlemma}\label{rest}
For $\overline{\ZZZ}(\kappa)$ of the second kind in $g=0$ and $B$ a boundary stratum of any dimension,
\begin{equation*}
\text{codim}_B(\overline{\ZZZ}(\kappa)\cap B)=\text{poles}(\kappa)-1
\end{equation*}
or the intersection is empty.
\end{restatable}
%%% This gives F-nef for all divisors  of this type %%%%%
Considering a divisor $Z_\kappa^n$ of the second kind in $\Pic_\QQ(\overline{\mathcal{M}}_{0,n})$ we apply this lemma to the intersection of the preimage of any $F$-curve under the forgetful morphism and $\overline{\ZZZ}(\kappa)$ to show that $Z_\kappa^n$ cannot contain any $F$-curve and hence necessarily has non-negative intersection will all such curves.
\begin{restatable}{thm}{Fnef}\label{Fnef}
Every irreducible component of the divisor $Z_\kappa^n$ of the second kind in $\Pic_\QQ(\overline{\mathcal{M}}_{0,n})$ is $F$-nef.
\end{restatable}
%%% Worth noting that this comes from just one section
Though ampleness is due to an abundance of sections, namely, enough sections to separate points and tangent vectors, the above theorem is the result of our explicit knowledge of the zero locus of just one section of the line bundles associated with each $Z_\kappa^n$.
%%% In the cases where we computed the classes this give nef %%%
In the cases considered in Theorem~\ref{thm:ZRdivI}, as discussed, our classes correspond to divisor classes considered by Fedorchuk~\cite{Fed} and shown to be nef. However, we include a proof of this fact from our perspective as it is instructive of a possible strategy to prove the remaining open cases\footnote{The strategy is well-known and applies to any subset of $F$-nef divisors that are boundary effective and closed under boundary restriction For a time it was considered a possible strategy to prove the $F$-conjecture until Pixton~\cite{Pixton} produced an example in $n=12$ of a nef divisor that is not boundary effective.}. Setting $p=$poles$(\kappa)$ we obtain the following.
\begin{restatable}{thm}{Nef}\label{Nef}
Every divisor $Z_\kappa^n$ of the second kind in $\Pic_\QQ(\overline{\mathcal{M}}_{0,n})$ for $\kappa=(d_1,\dots,d_n,1^{p-2})$ is nef.
\end{restatable}

The irreducible components of the divisors $Z_\kappa^n$ covered by Theorem~\ref{Nef} are further known to be $F$-nef by Theorem~\ref{Fnef}. The quesion of classifying the irreducible components of $Z_\kappa^n$ in genus $g=0$ is the classical question of classifying the connected components of certain Hurwitz spaces.
%\vspace{0.4cm}
%%%%%%%%%%%%%%%%%%%%%%%%%%%%%
%%%%%%%%%%%%%%%%%%%%%%%%%%%%%
\subsection{Irreducibility of certain Hurwitz spaces}
The question of the irreducibility of Hurwitz spaces can be viewed as a purely group theoretic question of the orbits of the conjugacy classes of the monodromy representations for covers $f:C\longrightarrow \PP^1$ under the action of the relevant subgroup of the braid group. From this perspective Clebsch~\cite{Clebsch} considered degree $d$ covers of $\PP^1$ where $\underline{\eta}$ contained only simple branching and showed using purely group theoretic methods that all conjugacy classes of monodromy representations of this type are equivalent under the action of the braid group that moves the branch points. Hence these spaces are irreducible. More recently,  Graber, Harris and Starr used similar methods to extend this result to degree $d$ simply ramified covers of genus $h\geq 1$ target curves. 

Kluitmann~\cite{Kluitmann} and Natanzon~\cite{Natanzon} used similar methods to show that the Hurwitz spaces with arbitrary genus source curve, genus $g=0$ target curve and simple branching at all but one branch point are connected. Liu and Osserman~\cite{LO} considered only the cases where the source curve has genus $g=0$ and used limit linear series to show the irreducibility of Hurwitz spaces with genus $g=0$ source and target curves and pure ramification. We obtain the following result.
\begin{restatable}{thm}{Hurwitz}\label{Hurwitz}
$\Hur(\underline{\eta})$ is irreducible if $\underline{\eta}=[\eta_1,\dots,\eta_s]$ specifying the cycle type of branching contains at most one $\eta_i$ that is not pure and at least $3g+(d-1)$ transpositions.
\end{restatable}
Our method relies on flat geometry and as such, we define the \emph{strata of exact abelian differentials with signature $\kappa=(k_1,...,k_n)$} as
\begin{equation*}
\HH_X(\kappa):=\{[C,\omega]\in \HH(\kappa)   \hspace{0.15cm}| \hspace{0.15cm}\int_\gamma\omega=0 \text{ for any closed path $\gamma$ in $C\setminus\{\text{poles of $\omega$}\}$}\},
\end{equation*}
and the \emph{stratum of exact canonical divisors with signature $\kappa=(k_1,...,k_n)$} as
\begin{equation*}
\X(\kappa):=\{[C,q_1,\dots,q_n]\in {\mathcal{M}}_{g,n}   \hspace{0.15cm}| \hspace{0.15cm}k_1q_1+...+k_nq_n\sim(\omega)\sim K_C \text{ and $[C,\omega]\in \HH_X(\kappa)$}\}.
\end{equation*}
Observe that $\X(\kappa)$ provides a finite cover of $\PP\HH_X(\kappa)$ due to the labelling of the zeros and poles. Further, if 
$$\kappa=(-p_1,\dots,-p_s,a_1,\dots,a_m)$$
for $p_i\geq 2$ and $a_i\geq 1$ and define
$$\underline{\eta}_\kappa=[(p_1-1)\cdots(p_s-1),(a_1+1),\dots,(a_m+1)]$$
where the entry $\eta_i$ of $\eta_\kappa$ specifies the cycle type of branching above a branch point. Then there exists a dominant morphism
$$\X(\kappa)\longrightarrow \Hur(\underline{\eta}_\kappa)$$
obtained by integrating the differential $\omega$ to obtain the cover of the projective line. This morphism factors through the action of the group permuting the markings appropriately for any $a_i=a_j$ and any $p_i=p_j$.

To prove Theorem~\ref{Hurwitz} we proceed by induction on $\underline{\eta}_\kappa$ via induction on the degree of $\Hur(\underline{\eta}_\kappa)$. The base case for the induction is provided by Clebsch~\cite{Clebsch} when $d=g+1$ and necessarily all ramification is simple. Let $\kappa=(-p_1,\dots,-p_s,a_1,\dots,a_n, 1^{3g+s-1})$ with $p_i\geq 2$ and $a_i\geq 1$. The degree of the covers obtained in $ \Hur(\underline{\eta}_\kappa)$ is given by
$$d=\sum_{i=1}^s(p_i-1).$$ 
Consider 
$$\phi:\X(\kappa)/S_{3g+s-1}\longrightarrow \mathcal{M}_{g,s+n}$$
that forgets all but the first $s+n$ points where $S_{3g+s-1}$ permutes the markings on the final $3g+s-1$ points. Lemma~\ref{dom} shows this morphism is dominant while Lemma~\ref{onecomp} shows that exactly one connected component of $\X(\kappa)/S_{3g+s-1}$ is dominant under $\phi$. Hence if $\X(\kappa)/S_{3g+s-1}$ is not irreducible there must be a connected component with positive dimensional fibres under $\phi$.  

Consider a point $[C,p_1,\dots,p_{s+n}]\in \mathcal{M}_{g,s+n}$ with one dimensional fibre under $\phi$. The forgotten points move on $C$ and hence as $C$ is one dimensional and irreducible, in the closure $\overline{\X}(\kappa)/S_{3g+s-1}$ for some simple zero must collide with each $p_1,\dots,p_{s}$, the marked poles.  By the degeneration of differentials of the second kind, this places $[C,p_1,\dots,p_{s+n}]$ inside the intersection of the image of $s$ different strata of lower degree assumed to be irreducible by the induction hypothesis. Lemma~\ref{nonproport} shows this intersection is of codimension at least two, necessitating fibres of dimension at least two contradicting the irreducibility of these lower degree strata and providing the inductive step.

The following is and example of an application of Theorem~\ref{Hurwitz}. 

\begin{restatable}{ex}{exgenusone}\label{exgenusone}
Consider the monodromy group
$$M=\{(12)(34),(456),(12),(13),(13),(14),(14),(35),(45),(56)\}$$
that gives a connected degree $d=5$ cover of the projective line by a genus $g=1$ curve. 
Theorem~\ref{Hurwitz} implies the associated Hurwitz space is irreducible.
\end{restatable}

Replacing the base case with the result of Liu and Osserman~\cite{LO} we obtain a sharpening of this result in the genus $g=0$ case.
\begin{restatable}{thm}{Hurwitzgenuszero}\label{Hurwitzgenuszero}
$\Hur(\underline{\eta})$ with source curve of genus $g=0$ is irreducible if $\underline{\eta}=(\eta_1,\dots,\eta_n)$ specifying the cycle type of branching contains at most one $\eta_i$ that is not pure and at least $d-3$ transpositions.
\end{restatable}
The following is an example of the application of this theorem. 

\begin{restatable}{ex}{exgenuszero}\label{exgenuszero}
Consider the monodromy group
$$M=\{(12)(34),(135),(132),(34),(45)\}$$
that gives a connected degree $d=5$ cover of the projective line by a genus $g=0$ curve. Theorem~\ref{Hurwitzgenuszero} implies the associated Hurwitz space is irreducible.
\end{restatable}

%%%%%%%%%%%%%%%%%%%%%%%%%%%%%
\begin{ack}
I am grateful to Ana-Maria Castravet, Dawei Chen, Maksym Fedorchuk, Quentin Gendron, Martin M\"{o}ller and Makoto Suwama for various useful discussions related to this circle of ideas and the anonymous referee for invaluable corrections and suggestions. The author was supported by the Alexander von Humboldt Foundation during the preparation of this article.
\end{ack}
%%%%%%%%%%%%%%%%%%%%%%%%%%%%%
%%%%%%%%%%%%%%%%%%%%%%%%%%%%%
%%%%%%%%%%%%%%%%%%%%%%%%%%%%%
\section{Preliminaries}

%%%%%%%%%%%%%%%%%%%%%%%%%%%%%%%
\subsection{Local charts for the strata of meromorphic differentials}\label{charts}
Any holomorphic flat surface with no vertical saddle connections can be obtained by the Veech or zippered rectangle construction~\cite{Veech}. Boissy~\cite{Boissy} showed how this generalises to the meromorphic case by the infinite zippered rectangle construction and provided a set of charts for $\HH(\kappa)$ in the meromorphic case. In this section we provide a brief summary of the construction of these charts and an example.

First we construct the basic domains that will be glued to obtain the charts. Let $\zeta=(v_1,...,v_n)\in \CC^n$ with Re$(v_i)>0$. Consider the following broken line $L$ in $\CC$:
\begin{itemize}
\item
the half-line $\RR^-$
\item
the concatenation of the $n$ vectors $v_i$ in order from the origin
\item
the horizontal half-line from the end of the above concatenation to the right.
\end{itemize}
We consider the subset $D^+(v_1,...,v_n)$ (or $D^-(v_1,...,v_n)$) as the set of complex numbers above (or below respectively) the broken line $L$. Similarly, if $n\geq 1$ we define $C^+(v_1,...,v_n)$ (or $C^-(v_1,...,v_n)$) as the set of complex numbers above (or below respectively) the concatenation of the $n$ vectors $v_i$ in order from the origin. The boundary of $C^+(v_1,...,v_n)$ or $C^-(v_1,...,v_n)$ will consist of two infinite vertical half-lines. Identifying the two half-lines we obtain an infinite half-cylinder with polygonal boundary. Figure 1 gives an example of how domains of this type can be glued together to obtain a flat surface in $\HH(2,1,-1,-2)$.
\begin{figure}[htbp]
\begin{center}
\begin{overpic}[width=0.45\textwidth]{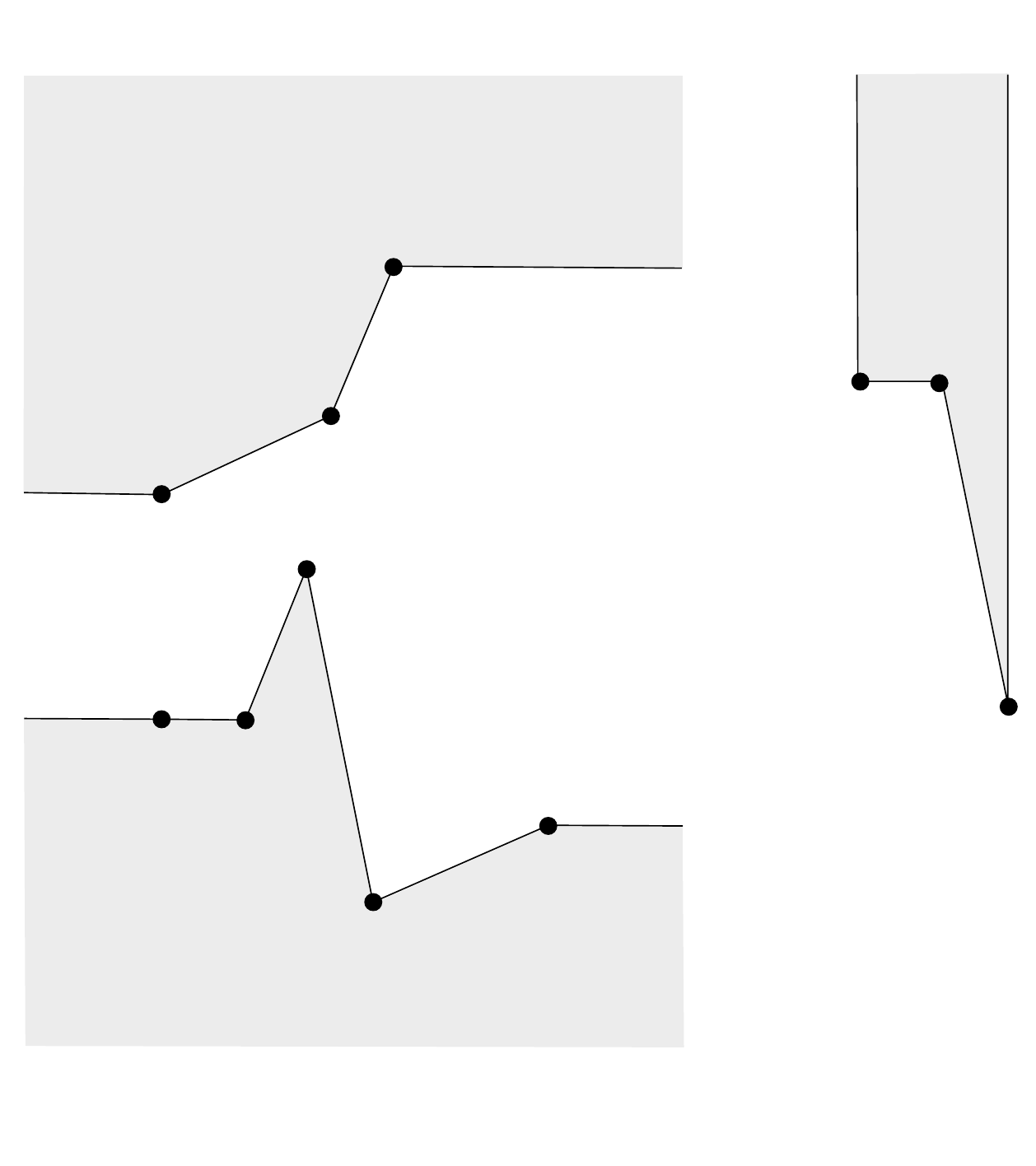}
%\color{green}

\put (22, 85){$D^+(v_4,v_2)$}
\put (6, 54){$l_1$}
\put (20, 58){$v_4$}
\put (32, 70){$v_2$}
\put (42, 74){$l_2$}

\put (20, 16){$D^-(v_1,v_2,v_3,v_4)$}
\put (6, 40){$l_1$}
\put (16, 40){$v_1$}
\put (20,45){$v_2$}
\put (30,38){$v_3$}
\put (38,28){$v_4$}
\put (51,31){$l_2$}

\put (72, 33){$C^+(v_1,v_3)$}
\put (70, 78){$l_3$}
\put (87.5, 78){$l_3$}
\put (75.5, 64.5){$v_1$}
\put (75.5, 64.5){$v_1$}
\put (79, 55){$v_3$}

\put (-20,-0){\footnotesize{Figure 1: Basic domains and gluing data of a flat surface in $\HH(2,1,-1,-2)$ }}
\end{overpic}
\end{center}
\end{figure}

Creating a set of charts for meromorphic stratum $\HH(\kappa)$ is simply a matter of giving a set of domains and gluing data.
For any $\zeta=(v_1,...,v_n)\in \CC^n$ with Re$(v_i)>0$ consider the following combinatorial data:
\begin{itemize}
\item
a collection $\bf{n}^+$ of integers $0=n_0^+\leq n_1^+\leq...\leq n_d^+<...<n_{d+s^+}^+=n$
\item
a collection $\bf{n}^-$ of integers $0=n_0^-\leq n_1^-\leq...\leq n_d^-<...<n_{d+s^-}^-=n$
\item
a pair of permutations $\pi_t,\pi_b\in S_n$
\item
a collection $\bf{d}$ of integers $0=d_0<d_1<d_2<...<d_r=d$.
\end{itemize}
For $i\in\{0,...,d-1\}$ define basic domains $D_i^+(v_{\pi_t(n_i^++1)},...,v_{\pi_t(n^+_{i+1})})$ and  $D_i^-(v_{\pi_b(n_i^++1)},...,v_{\pi_b(n^+_{i+1})})$. For $i\in\{d,...,d+s^+-1\}$ define basic domains $C_i^+(v_{\pi_t(n_i^++1)},...,v_{\pi_t(n^+_{i+1})})$  and for $i\in\{d,...,d+s^--1\}$ define basic domains $C_i^-(v_{\pi_b(n_i^++1)},...,v_{\pi_b(n^+_{i+1})})$.

Now glue these domains as follows:
\begin{itemize}
\item
each vector corresponding to a $v_i$ in a $+$domain is glued to the corresponding vector in a $-$domain
\item
for each $C_i^+$ and $C_i^-$ the two vertical lines are glued
\item
each left line of a domain $D^+_i$ is glued to the left line of the domain $D_i^-$
\item
for each $i\in \{1,...,d\}\setminus \{d_1,...,d_r\}$ the right line of domain $D^-_i$ is glued to the right line of $D_{i+1}^+$
\item
for $i=d_k$, $k>0$, the right line of the domain $D_i^-$ is glued to the right line of $D^-_{d_{k+1}-1}$.
\end{itemize}
Considering only the combinatorial data that gives a connected surface, we obtain a flat surface with $s=s^++s^-$ simple poles and $r$ non-simple poles of order $d_i+1$ for $i=1,...,r$. The type of conical singularities, or the order of the zeros, will be determined by the combinatorial data. For example, the flat surface given in Figure 1 can be given as the vector $(v_1,v_2,v_3,v_4)$ and combinatorial data $d=1, s^+=1,s^-=0, {\bf n}^+=(2,4),{\bf n}^-=(4), \pi_t=(4,2,1,3), \pi_b= (1,2,3,4)$ and ${\bf d}=(0,1)$.

The parameter $\zeta$ is uniquely defined and the saddle connections form a basis of the relative homology $H_1(S,\sum,\ZZ)$ where $\sum$ is the set of conical singularities or zeros of the differential. Hence $n=2g+|\kappa|-2$. Boissy~\cite[Proposition 3.7]{Boissy} showed that analogous to the holomorphic case, any meromorphic translation surface with no vertical saddle connection is obtained by the infinite zippered rectangle construction. But on any meromorphic flat surface the set of saddle connections is at most countable. Hence it is always possible to rotate a flat surface so that there are no vertical saddle connections. Hence up to possible rotation, the infinite zippered rectangle construction provides a set of local charts for meromorphic stratum $\HH(\kappa)$.

%%%%%%%%%%%%%%%%%%%%%%%%%%%%%%\
\subsection{Divisor theory on $\Mgnb$}
The first Chern class $\lambda$, of the Hodge bundle on $\Mgnb$ and the first Chern class $\psi_i$, of the cotangent bundle on $\Mgnb$ associated with the $i$th marked point for $1\leq i\leq n$ give extensions of the classes that generate $\Pic_\QQ(\Mgn)$. 

The boundary $\partial\Mgnb=\Mgnb-\Mgn$ of the compactification is codimension one. Let $\Delta_0$ be the locus of curves in $\Mgnb$ with a nonseparating node and $\Delta_{i:S}$ for $0\leq i\leq g$, $S\subset \{1,\dots,n\}$ be the locus of curves with a separating node that separates the curve such that one of the components has genus $i$ and contains precisely the marked points in $S$. Hence we require $|S|\geq 2$ for $i=0$ and $|S|\leq n-2$ for $i=g$ and observe that $\Delta_{i:S}=\Delta_{g-i:S^c}$. The boundary divisors are irreducible and we denote the class of $\Delta_{i:S}$ in $\Pic_\QQ(\Mgnb)$ by $\delta_{i:S}$. See~\cite{AC}\cite{HarrisMorrison} for more information.

For $g\geq 3$, these divisors freely generate $\Pic_\QQ(\Mgnb)$. 

For $g=2$, the classes $\lambda$, $\delta_0$ and $\delta_1$ generate $\Pic_\QQ(\overline{\mathcal{M}}_2)$ with the relation
\begin{equation*}
\lambda=\frac{1}{10}\delta_0+\frac{1}{5}\delta_1.
\end{equation*}
Further, $\lambda,\delta_0,\psi_i$ and $\delta_{i:S}$ generate $\Pic_\QQ(\overline{\mathcal{M}}_{2,n})$ with this relation pulled back under the morphism $\varphi:\overline{\mathcal{M}}_{2,n}\longrightarrow \overline{\mathcal{M}}_{2}$ that forgets the $n$ marked points.

For $g=1$ there are two further relations in $\Pic_\QQ(\overline{\mathcal{M}}_{1,n})$
\begin{equation*}
\lambda=\frac{1}{12}\delta_0=\psi_i-\sum_{i\in S}\delta_{0:S}
\end{equation*}
for $1\leq i\leq n$.

For $g=0$ the relations in $\Pic_\QQ(\overline{\mathcal{M}}_{0,n})$ are
\begin{equation*}
\psi_i+\psi_j=\sum_{i\in S,j\nin S}\delta_{0:S}
\end{equation*}
for any $i\ne j$.

%%%%%%%%%%%%%%%%%%%%%%%%%%%%%%\
\subsection{Maps between moduli spaces}\label{maps}
Maps between moduli spaces are a useful tool in the computation of divisor classes. For a fixed general $[X,q,q_1]\in \mathcal{M}_{h,2}$ consider the map 
\begin{eqnarray*}
\begin{array}{cccc}
\pi_h:&\overline{\mathcal{M}}_{g,n}&\rightarrow& \overline{\mathcal{M}}_{g+h,n}\\
&[C,p_1,\dots,p_n]&\mapsto&[C\bigcup_{p_1=q}X,q_1,p_2,\dots,p_n].
\end{array}
\end{eqnarray*} 
that glues points $p_1$ and $q$ to form a node. The pullback of the generators of $\Pic_\QQ(\overline{\mathcal{M}}_{g+h,n})$ are presented in~\cite{AC}
\begin{equation*}
\pi_h^*\lambda=\lambda, \hspace{0.5cm}\pi_h^*\delta_0=\delta_0,\hspace{0.5cm}\pi_h^*\delta_{h:\{1\}}=-\psi_1
\end{equation*}
and
\begin{equation*}
\pi_h^*\psi_i=\begin{cases}
0 &\text{ for i=1}\\
\psi_i &\text{ for $i=2,\dots,n$}
\end{cases}
\end{equation*}
and for $1\in S$
\begin{equation*}
\pi_h^*\delta_{i:S}=\begin{cases}
0 &\text{ for $i<h$}\\
-\psi_1 &\text{ for $i=h,S=\{1\}$}\\
\delta_{i-h:S}&\text{ otherwise.}\\
\end{cases}
\end{equation*}

%%%%%%%%%%%%%%%%%%%%%%%%%%%%%
%%%%%%%%%%%%%%%%%%%%%%%%%%%%%
\section{Zero residue strata of differentials}
%%%%%%%%%%%%%%%%%%%%%%%%%%%%%
\subsection{Residue conditions}\label{defnzero}
Fix
\begin{equation*}
\kappa=(k_1,...,k_m,-p_1,...,-p_r,(-1)^s),
\end{equation*} 
with $k_i>0$ and $p_i>1$ and consider a chart as described in \S\ref{charts} for $\HH(\kappa)$ 
given by ${\bf n}^+,{\bf n}^- ,\pi_t,\pi_b, s^+,s^-$ and ${\bf d}$ where $d_0=0$ and $d_i=\sum_{j=1}^i(p_j-1)$ for $i\geq 1$. The residue at the pole $-p_i$ is given by
\begin{equation*}
\displaystyle{\frac{1}{2\pi i} \sum_{\tiny{j=n^+_{d_{i-1}}+1}}^{n^+_{d_i}}\left( v_{\pi_t(j)}   -v_{\pi_b(j)}   \right),}
\end{equation*} 
where in such equations the saddle connections $v_i$ refers to the length of the saddle connection under the flat metric. This formula is obtained by integrating the differential around an anti-clockwise closed loop enclosing the pole of interest.  Hence the residue at a simple pole in a basic domains $C^+(v_1,...,v_m)$ or $C^-(v_1,...,v_m)$ are given by $(1/2\pi i)(v_1+...+v_m)$ and $(1/2\pi i)(-v_1-...-v_m)$ respectively.

For example, from Figure 1 we see that the residue at the pole of order $2$ is
\begin{equation*}
\frac{1}{2\pi i}(v_4+v_2-v_4-v_3-v_2-v_1)=\frac{1}{2\pi i}(-v_1-v_3),
\end{equation*}
while the residue at the simple pole is
\begin{equation*}
\frac{1}{2\pi i}(v_1+v_3)
\end{equation*}
and as expected the sum of the residues is zero. Observe that the sum of the residues across all poles on any surface will be zero as each vector appears once in a $+$domain and once in a $-$domain.

\begin{defn}
The \emph{stratum of zero residue abelian differentials with signature $\kappa$} is defined as
\begin{equation*}
\HH_Z(\kappa):=\{[C,\omega]\in \HH(\kappa)   \hspace{0.15cm}| \hspace{0.15cm}\res_{p}(\omega)=0\text{ for all $p$ non-simple poles of $\omega$}\}.
\end{equation*}
Further we define the \emph{stratum of zero residue canonical divisors with signature $\kappa$} as
\begin{equation*}
\ZZZ(\kappa):=\{[C,p_1,...,p_m]\in {\mathcal{M}}_{g,m}   \hspace{0.15cm}| \hspace{0.15cm}k_1p_1+...+k_mp_m\sim(\omega)\sim K_C \text{ and $\res_{p_i}(\omega)=0$ for $k_i\leq -2$}\}.
\end{equation*}
\end{defn}

\begin{ex}
Consider $\kappa$ with a single pole. The pole is necessarily non-simple and we have $\HH_Z(\kappa)=\HH(\kappa)$.
\end{ex}

\begin{ex}
Consider $\kappa$ with two or more poles including exactly one simple pole. For any surface in $\HH(\kappa)$ the sum of the residues at the non-simple poles must add to the non-zero residue at the simple pole to give zero. Hence $\HH_Z(\kappa)$ is empty. 
\end{ex}

\begin{ex}\label{g0res}
Consider a genus $g=0$ surface in $\HH(2,-2,-2)$. The resulting differential is given locally by $c(1-z)^2z^{-2}dz$ for some constant $c\in \CC^*$. Hence the residue at $0$ and $\infty$ are necessarily non-zero. Alternatively, observe $\dim\HH(2,-2,-2)=1$. Hence any surface appears in a one dimensional chart with $\zeta=v_1\in \CC$ with Re$(z_1)>0$. Necessarily $v_1$ must appear in the the basic domains for both poles for the surface to be connected and hence the residues equal $\pm v_1/2\pi i\ne 0$. From one final perspective, on a genus $g=0$ surface with one conical singularity the length of any saddle connection is obtained by integrating the differential along a closed path. If all residues are zero then the length is necessarily zero providing a contradiction. 
\end{ex}

For $\kappa=(k_1,...,k_m,-p_1,...,-p_r,(-1)^k)$, with $p_i\geq 2$ and $k_i\geq 1$ the image of the residue map 
\begin{equation*}
\res: \HH(\kappa)\longrightarrow \CC^r\times (\CC^*)^k
\end{equation*}
was completely classified by Gendron and Tahar~\cite{{GendronTahar}}. In our case their results provide the following.

\begin{prop}\label{existence}
$\HH_Z(\kappa)$ is empty if and only if 
\begin{enumerate}
\item
$k=1$; or
\item
$g=0, k=0$ and $k_i\geq \sum_{j=1}^r(p_j-1)$ for some $i$.
\end{enumerate}
\end{prop}

\begin{prop}\label{affinemanifold}
If $\kappa$ does not satisfy $(a)$ and $(b)$ above, then $\HH_Z(\kappa)$ is cut out in $\HH(\kappa)$ by linear equations in period coordinates with real coefficients and zero constant term and 
\begin{equation*}
\text{codim}(\HH_Z(\kappa))=\begin{cases}
r-1 &\text{ for $k=0$}\\
r &\text{ for $k\geq 2$}.
\end{cases}
\end{equation*} 
\end{prop}

\begin{proof}
Any surface in $\HH_Z(\kappa)$ lies in a chart of $\HH(\kappa)$ described in \S\ref{charts}. The residue at each pole is given in this chart by a linear equation in period coordinates with coefficients $\pm1/2\pi i $. If a coordinate appears in a $+$domain and a $-$domain for the same pole it will not appear in a residue equation. Hence the coordinates that appear in the residue equations are the coordinates that connect the surface and any relation in the residues would violate the connectedness of the surface. 
\end{proof}

%%%%%%%%%%%%%%%%%%%%%%%%%%%%%
%%%%%%%%%%%%%%%%%%%%%%%%%%%%%
%%%%%%%%%%%%%%%%%%%%%%%%%%%%%
\section{Degeneration of zero residue conditions}\label{degen}
A pointed curve $[C,p_1,...,p_n]\in\Mgn$ is contained in \emph{the stratum of canonical divisors of type} $\kappa=(k_1,...,k_n)$ denoted ${\PPP}(\kappa)$, if 
\begin{equation*}
\sum_{i}^nk_ip_i\sim K_C.
\end{equation*}
A stable pointed curve $[C,p_1,...,p_n]\in\Mgnb$ is contained in \emph{the moduli space of twisted canonical divisors of type} $\kappa=(k_1,...,k_n)$ defined by Farkas and Pandharipande~\cite{FP}, denoted $\widetilde{\PPP}(\kappa)$, if there exists a collection of (possibly meromorphic) non-zero differentials $\{\eta_j\}$ on the irreducible components $C_j$ of $C$ with $(\eta_j)\sim D_j\sim K_{C_j}$ such that
\begin{enumerate}
\item
The support of $D_j$ contains the set of marked points and the nodes lying in $C_j$, moreover if $p_i\in C_j$ then $\ord_{p_i}(D_j)=k_i$.
\item
If $q$ is a node of $C$ and $q\in C_i\cap C_j$ then $\ord_q(D_i)+\ord_q(D_j)=-2$.
\item
If $q$ is a node of $C$ and $q\in C_i\cap C_j$ such that $\ord_q(D_i)=\ord_q(D_j)=-1$ then for any $q'\in C_i\cap C_j$, we have $\ord_{q'}(D_i)=\ord_{q'}(D_j)=-1$. We write $C_i\sim C_j$.
\item
If $q$ is a node of $C$ and $q\in C_i\cap C_j$ such that $\ord_q(D_i)>\ord_q(D_j)$ then for any $q'\in C_i\cap C_j$ we have $\ord_{q'}(D_i)>\ord_{q'}(D_j)$. We write $C_i\succ C_j$.
\item
There does not exist a directed loop $C_1\succeq C_2\succeq...\succeq C_k\succeq C_1$ unless all $\succeq$ are $\sim$.
\end{enumerate}  
In particular, the last condition implies that if $q$ is a node of $C$ and $q\in C_i\cap C_j$ such that $\ord_q(D_i)=\ord_q(D_j)=-1$ then for any $q'\in C_i\cap C_j$ we have $\ord_{q'}(D_i)=\ord_{q'}(D_j)=-1$. Farkas and Pandharipande showed that in addition to the main component $\overline{\PPP}(\kappa)$ containing $\PPP(\kappa)$, this space contained extra components completely contained in the boundary of the moduli space.
Bainbridge, Chen, Gendron, Grushevsky and M\"oller~\cite{BCGGM1} provided the condition that a twisted canonical divisor lies in the main component. Let $\Gamma$ be the dual graph of $C$. A twisted canonical divisor of type $\kappa$ is the limit of twisted canonical divisors on smooth curves if there exists a twisted differential $\{\eta_j\}$ on $C$ such that
\begin{enumerate}
\item
If $q$ is a node of $C$ and $q\in C_i\cap C_j$ such that $\ord_q(D_i)=\ord_q(D_j)=-1$ then $\res_q(\eta_i)+\res_q(\eta_j)=0$; and
\item
there exists a full order on the dual graph $\Gamma$, written as a level graph $\overline{\Gamma}$, agreeing with the order of $\sim$ and $\succ$, such that for any level $L$ and any connected component $Y$ of  $\overline{\Gamma}_{>L}$ that does not contain a prescribed pole we have
\begin{equation*}
\sum_{\begin{array}{cc}\text{level}(q)=L, \\q\in C_i\subset Y\end{array}}\res_{q}(\eta_i)=0
\end{equation*} 
\end{enumerate}
Condition $(b)$ is known as the \emph{global residue condition}.

The question for us is the conditions on the twisted differential such that the the twisted canonical divisor is a limit of smooth canonical divisors with conditions on the residues. We record the result as the following theorem.

\ZeroResidueGRC*

\begin{proof}
This proof follows by simple alteration of the proof the Global Residue Theorem~\cite{BCGGM1} that we outline here.

The proof that the conditions are necessary follows from the proof of the necessity of the Global Residue Theorem~\cite[\S 4.1]{BCGGM1}. Part (a) follows directly, while part (c) follows by the alteration to this proof that we instead let $Y$ be a connected component of $X_{>L}$ that contains no marked non-zeroed poles (rather than no marked poles). The  surfaces $Y_t$ then contain no prescribed non-zeroed pole for $t\ne0$ and hence equations $(4.3)$ and $(4.4)$ and the rest of the proof follow verbatim. Part (b) follow by a minor alteration of this argument to consider the residues at zeroed poles on components of the twisted differential in the limit.

The proof that the conditions are sufficient also follows by a modification of either of the two proofs of smoothing presented in~\cite{BCGGM1}. For example, consider the flat geometric proof of smoothing~\cite[\S 5]{BCGGM1}. The surgery in the upper level surface in the slit residue construction can be made to maintain residues at zeroed poles in higher level provided not all marked poles in higher level are zeroed. Further, it is clear that maintaining zero residues at all marked poles in higher levels is possible via the slit residue construction if condition (c) is satisfied. Finally, observe that the slit residue construction maintains any zero residues at poles on the lower level after smoothing. Hence the smoothing will satisfy condition (b). 
\end{proof}

%%%%%%%%%%%%%%%%%%%%%%%%%%%%%
%%%%%%%%%%%%%%%%%%%%%%%%%%%%%
%%%%%%%%%%%%%%%%%%%%%%%%%%%%%
\section{Divisor class computation}\label{class}
In this section we compute the class of divisors in $\Mgnb$ coming from the strata of zero residue abelian differentials. In~\cite{Mullane2} the author used the method of maps between moduli spaces to compute many classes of divisors $D^n_\kappa$  coming from the strata of abelian differentials. Consider the signature
\begin{equation*}
\kappa=(k_1,...,k_n,f_1^{\alpha_1},...,f_m^{\alpha_m})
\end{equation*}
where $f_i\ne f_j$ for $i\ne j$. The divisor $D^n_\kappa$  in $\Mgnb$ is defined as:
\begin{equation*}
D^n_\kappa:=\frac{1}{\alpha_1!...\alpha_m!}\varphi_*\overline{\PPP}(\kappa)
\end{equation*}
where $\varphi$ forgets the last $\sum\alpha_i$ marked points equal to $g-2$ or $g-1$ for $\kappa$ a holomorphic or meromorphic signature respectively. The method of class computation in~\cite{Mullane2} extends to our situation where we utilise maps between moduli spaces and the theory of the degeneration of abelian differentials~\cite{BCGGM1} to pullback a known divisor of type $D^n_\kappa$ coming from the strata of meromorphic differentials to obtain classes coming from the strata of zero residue abelian differentials. In~\cite{Mullane2} the author presents a detailed exposition of this method. We begin with a definition of the divisors of interest. 

\begin{defn}\label{divdefn}
Let $j,k$ be the number of non-simple and simple poles respectively in meromorphic signature 
\begin{equation*}
\kappa=(k_1,...,k_n,f_1^{\alpha_1},...,f_m^{\alpha_m})
\end{equation*}
where $f_i\ne f_j$ for $i\ne j$. For $j\geq 1$ with $|\kappa|=g-2+n+j$ if $k=0$ and $|\kappa|=g-1+n+j$ if $k\geq 2$. Then $Z^n_\kappa$ for $n\geq 1$ is the divisor in $\Mgnb$ defined as:
\begin{equation*}
Z^n_\kappa:=\frac{1}{\alpha_1!...\alpha_m!}\varphi_*\overline{\ZZZ}(\kappa)
\end{equation*}
where $\varphi$ forgets the last $\sum\alpha_i$ marked points equal to $g-2+j$ or $g-1+j$ for $k=0$ or $k\geq 2$ respectively.
\end{defn}

Proposition~\ref{existence} states when $\ZZZ(\kappa)$ is empty. Hence outside of the exceptions in genus $g=0$, Proposition~\ref{affinemanifold} shows $\overline{\ZZZ}(\kappa)$ forms a codimension $j-1$ or $j$ subvariety of $\overline{\PPP}(\kappa)$ for $k=0$ or $k\geq 2$ respectively. As $\overline{\PPP}(\kappa)$ is codimension $g$ in $\Mbar{g}{|\kappa|}$ we obtain the image of $\overline{\ZZZ}(\kappa)$ under the pushforward of the forgetful morphism is a divisor\footnote{Recall that the image of any irreducible subvariety contracted by $\varphi$ is defined to be trivial under the pushforward map $\varphi_*$.}.

Before computing classes, we introduce a well-known method of producing expressions for the trivial divisor in genus $g=0$ which will allow us to simplify our class expressions in this case. For any complete graph $\Gamma$ on $n$ vertices with edges labels in $\QQ$, define the edge function $e(i\sim j)$ to return the label of the unique edge between $i$ and $j$. Define
\begin{equation*}
e(i)=\sum_{j\ne i}e(i\sim j)
\end{equation*}
and
\begin{equation*}
e(S)=\sum_{i\in S,j\nin S}e(i\sim j)
\end{equation*}
for any $S\subset \{1,\dots,n\}$. Then any such edge labelled complete graph $\Gamma$ defines a divisor class in $\Pic_\QQ(\MOnb)$,
\begin{equation*}
D(\Gamma)=\sum_{i=1}^n e(i)\psi_i - \sum_{1\in S} e(S)\delta_{0:S}
\end{equation*}
 \begin{lem}\label{trivial}
 $D(\Gamma)$ is trivial.
 \end{lem}
 
 \begin{proof}
 Keel~\cite{Keel} gives the relation
 \begin{equation*}
 \psi_i+\psi_j=\sum_{i\in S,j\nin S}\delta_{0:S}
 \end{equation*}
 for any $i\neq j$. Hence
 \begin{eqnarray*}
 \sum_{1\leq i<j\leq n}e(i\sim j)\left(\psi_i+\psi_j  \right)&=& \sum_{1\leq i<j\leq n}e(i\sim j)\sum_{i\in S,j\nin S}\delta_{0:S}\\
 \sum_{i=1}^n e(i)\psi_i&=& \sum_{1\in S}e(S)\delta_{0:S}
 \end{eqnarray*}
 \end{proof}

\begin{restatable}{thm}{ZRdivI}
\label{thm:ZRdivI}
Consider $\underline{d}=(d_1,...,d_n)$ with $-1\nin \underline{d}$ and $|\{d_i<0\}|=m$ such that $\sum d_i=g-m$. 
For $g\geq 2$
\begin{equation*}
Z^{n}_{\underline{d},1^{g+m-2}}=-\lambda+\sum_{j=1}^n\begin{pmatrix}d_j+1\\2  \end{pmatrix}\psi_j
-0\cdot\delta_0-\sum_{\text{all poles }\nin S}\begin{pmatrix}|d_S+|S^-|-i|+1\\2  \end{pmatrix}\delta_{i:S}.
\end{equation*}
For $g=1$ 
\begin{equation*}
Z^{n}_{\underline{d},1^{m-1}}=\left(-1+\sum_{j=1}^n\begin{pmatrix}d_j+1\\2  \end{pmatrix}\right)\lambda
-\sum_{|S|\geq 2}\left(\begin{pmatrix}|d_S+|S^-||+1\\2  \end{pmatrix}-\sum_{j\in S}\begin{pmatrix}d_j+1\\2  \end{pmatrix}\right)\delta_{0:S}
\end{equation*}
where $S^-=\{i\in S\hspace{0.1cm}|\hspace{0.1cm}d_i<0\}$ and $d_S=\sum_{i\in S}d_i$. 
For $g=0$ 
\begin{equation*}
Z^{n}_{\underline{d},1^{m-2}}=\sum_{j=1}^n f(j)\psi_j-\sum_{1\in S}f(S) \delta_{0:S}
\end{equation*}
where 
$$f(S)=\frac{1}{2}\big|\sum_{i\in S}d_i+|S^-|  \big|.$$
Further, the divisors in $g=0$ are boundary effective.
\end{restatable}

\begin{proof}
Observe that when $m=1$ we have
\begin{equation*}
Z^{n}_{\underline{d},1^{g+m-2}}=D^n_{\underline{d},1^{g-1}}
\end{equation*}
as necessarily the isolated pole must have residue zero. This matches the formula for this class given by M\"uller~\cite{Muller} and Grushevsky and Zakharov~\cite{GruZak}. 

Now assume $m\geq 2$, and without loss of generality, assume $d_1\leq -2$. For a fixed general $[X,q,q_1]\in \mathcal{M}_{h,2}$ consider the map 
\begin{eqnarray*}
\begin{array}{cccc}
\pi_h:&\overline{\mathcal{M}}_{g,n}&\rightarrow& \overline{\mathcal{M}}_{g+h,n}\\
&[C,p_1,\dots,p_n]&\mapsto&[C\bigcup_{p_1=q}X,q_1,p_2,\dots,p_n].
\end{array}
\end{eqnarray*} 
that glues points $p_1$ and $q$ to form a node. Setting $h=-d_1$, by Theorem~\ref{ZRGRC} we have
\begin{equation*}
Z^n_{\underline{d},1^{g+m-2}}=\pi_h^*Z^n_{1,d_2,\dots,d_n,1^{g+h+(m-1)-2}}.
\end{equation*}
Inductively computing the class from the known base cases of $m=1$ provides the result. For $g=0$ this gives
\begin{equation*}
Z^{n}_{\underline{d},1^{m-2}}=\sum_{j=1}^n\begin{pmatrix}d_j+1\\2  \end{pmatrix}\psi_j-\sum_{1\in S}\begin{pmatrix}|d_S+|S^-||+1\\2  \end{pmatrix}   \delta_{0:S}.
\end{equation*}
To show this expression is in fact boundary effective we appeal to Lemma~\ref{trivial}. Let 
\begin{equation*}
f_i=\begin{cases}
d_i+1 &\text{ for $d_i<0$}\\
d_i &\text{ otherwise.}
\end{cases}
\end{equation*}
Hence $\sum f_i=0$. Now consider $\Gamma_1$ the complete graph on $n$ vertices with $e(i\sim j)=-f_if_j$. We have
\begin{eqnarray*}
e(i)&=&-f_i\left(\sum_{j\ne i}f_j  \right)=f_i^2\\
e(S)&=&-\left(\sum_{i\in S}f_i  \right)\left(\sum_{i\nin S}f_i  \right)=(d_S+|S^-|)^2
\end{eqnarray*}
Hence 
\begin{equation*}
D(\Gamma_1)=\sum_{i=1}^nf_i^2\psi_i-\sum_{1\in S}(d_S+|S^-|)^2\delta_{0:S}
\end{equation*}
is trivial and
\begin{equation*}
2Z^{n}_{\underline{d},1^{m-2}}-D(\Gamma_1)=\sum_{i=1}^n|f_i|\psi_i-\sum_{1\in S}\left|\sum_{i\in S} f_i\right|\delta_{0:S}
\end{equation*}
Now consider $\Gamma_2$ the complete graph on $n$ vertices with 
\begin{equation*}
e(i\sim j)=\begin{cases}
-f_if_j &\text{ if $-f_if_j>0$}\\
0 &\text{ otherwise.}
\end{cases}
\end{equation*}
For any $S\subset\{1,\dots,n\}$ define $S^-=\{i\in S|f_i<0\}$ and $S^+=\{i\in S|f_i\geq 0\}$. Further, let
\begin{equation*}
\sum_{f_i>0}f_i=N, \hspace{0.3cm} f_S^+=\sum_{i\in S^-}d_i, \hspace{0.3cm} f_S^-=\left|\sum_{i\in S^-}d_i  \right|.
\end{equation*}
We obtain
\begin{eqnarray*}
e(i)&=&N|f_i|\\
e(S)&=&f_S^-\left(N-f_S^+ \right)+f_S^+\left( N-f_S^-   \right).
\end{eqnarray*}
Hence $Z^{n}_{\underline{d},1^{m-2}}$ is boundary effective if $e(S)\geq |f_S|=|f_S^+-f_S^-|$ and we are left to show that for any $0\leq a,b\leq N$ the inequality
\begin{equation*}
a(N-b)+b(N-a)\geq N|a-b|
\end{equation*}
holds. Without loss of generality we assume $a\geq b$, then the inequality follows as $bN\geq ab$.
\end{proof}

\begin{restatable}{thm}{ZRdivII}
\label{thm:ZRdivII}
Consider $\underline{d}=(d_1,...,d_n)$ where $-1$ has multiplicity $k\geq 2$ in $\underline{d}$ and $|\{d_i<0\}|=m+k$ such that $\sum d_i=g-m-1$ we have for $g\geq 2$
\begin{equation*}
Z^{n}_{\underline{d},1^{g+m-1}}=-\lambda+\sum_{j=1}^n\begin{pmatrix}d_j+1\\2  \end{pmatrix}\psi_j
%-0\cdot\delta_0
-\sum_{|S^\text{s}|=0}\begin{pmatrix}|d_S+|S^-|-i|+1\\2  \end{pmatrix}\delta_{i:S}
-\sum_{i=1}^g\sum_{|S^\text{s}|\ne0,k }\begin{pmatrix}d_S+|S^\text{ns}|-i+1\\2  \end{pmatrix}\delta_{i:S}.
\end{equation*}
For $g=1$ this gives
\begin{eqnarray*}
&&Z^{n}_{\underline{d},1^{m}}=\left(-1+\sum_{j=1}^n\begin{pmatrix}d_j+1\\2  \end{pmatrix}\right)\lambda
-\sum_{|S^\text{s}|=0}\left(\begin{pmatrix}|d_S+|S^-||+1\\2  \end{pmatrix}-\sum_{j\in S}\begin{pmatrix}d_j+1\\2  \end{pmatrix}\right)\delta_{0:S}\\
&&-\sum_{|S^\text{s}|=k}\left(\begin{pmatrix}|d_S+|S^-|+1|+1\\2  \end{pmatrix}-\sum_{j\in S}\begin{pmatrix}d_j+1\\2  \end{pmatrix}\right)\delta_{0:S}
-\sum_{|S^\text{s}|\ne0,k}\left(\begin{pmatrix}d_S+|S^\text{ns}|+k-1\\2  \end{pmatrix}   -\sum_{j\in S}\begin{pmatrix}d_j+1\\2  \end{pmatrix}\right)\delta_{0:S}.
\end{eqnarray*}
For $g=0$ this gives
\begin{equation*}
Z_{\underline{d}}=\sum_{d_j\leq-2}|d_j+1|\psi_j
%-0\cdot\delta_0
-\sum_{\tiny{\begin{array}{cc}|S^\text{s}|=0\\ d_S+|S^-|<0 \end{array}}}|d_S+|S^-||\delta_{0:S}
\end{equation*}
where $S^-=\{i\in S\hspace{0.1cm}|\hspace{0.1cm}d_i<0\}$, $S^\text{ns}=\{i\in S \hspace{0.1cm}|\hspace{0.1cm}d_i\leq-2\}$ (the non-simple poles in $S$), $S^\text{s}=\{i\in S \hspace{0.1cm}|\hspace{0.1cm}d_i=-1\}$ (the simple poles in $S$) and $d_S=\sum_{i\in S}d_i$. %Further, in $g=0$ these divisors are boundary effective.
\end{restatable}

\begin{proof}
Again, observe that when $m=0$ we have
\begin{equation*}
Z^{n}_{\underline{d},1^{g+m-1}}=D^n_{\underline{d},1^{g-1}}
\end{equation*}
and the formula matches the class given by M\"uller~\cite{Muller} and Grushevsky and Zakharov~\cite{GruZak}. 

Now assume $m\geq 1$, and without loss of generality, assume $d_1\leq -2$. For a fixed general $[X,q,q_1]\in \mathcal{M}_{h,2}$ consider the map 
\begin{eqnarray*}
\begin{array}{cccc}
\pi_h:&\overline{\mathcal{M}}_{g,n}&\rightarrow& \overline{\mathcal{M}}_{g+h,n}\\
&[C,p_1,\dots,p_n]&\mapsto&[C\bigcup_{p_1=q}X,q_1,p_2,\dots,p_n].
\end{array}
\end{eqnarray*} 
that glues points $p_1$ and $q$ to form a node. Setting $h=-d_1$, by Theorem~\ref{ZRGRC} we have
\begin{equation*}
Z^n_{\underline{d},1^{g+m-1}}=\pi_h^*Z^n_{1,d_2,\dots,d_n,1^{g+h+(m-1)-1}}.
\end{equation*}
Inductively computing the class from the known base cases of $m=0$ provides the result. 

For $g=0$ this process gives
\begin{equation*}
Z_{\underline{d}}=\sum_{j=1}^n\begin{pmatrix}d_j+1\\2  \end{pmatrix}\psi_j
%-0\cdot\delta_0
-\sum_{\tiny{\begin{array}{cc}|S^\text{s}|=0 \end{array}}}\begin{pmatrix}|d_S+|S^\text{ns}||+1\\2  \end{pmatrix}\delta_{0:S}
-\sum_{\tiny{\begin{array}{cc}1\in S\\|S^\text{s}|\ne 0,k \end{array}}}\begin{pmatrix}d_S+|S^\text{ns}|+1\\2  \end{pmatrix}\delta_{0:S}
\end{equation*}
where $S^\text{ns}=\{i\in S |d_i\leq-2\}$ (the non-simple poles in $S$) and $S^\text{s}=\{i\in S |d_i=-1\}$ (the simple poles in $S$). 

To simplify this expression we appeal to Lemma~\ref{trivial}. Let 
\begin{equation*}
f_i=\begin{cases}
d_i+1 &\text{ for $d_i\leq-2$}\\
d_i &\text{ otherwise.}
\end{cases}
\end{equation*}
Hence $\sum f_i=-1$. Now consider $\Gamma$ the complete graph on $n$ vertices with $e(i\sim j)=-f_if_j$. We have
\begin{eqnarray*}
e(i)&=&-f_i\left(\sum_{j\ne i}f_j  \right)=f_i(f_i+1)\\
e(S)&=&-\left(\sum_{i\in S}f_i  \right)\left(\sum_{i\nin S}f_i  \right)=(d_S+|S^\text{ns}|)(d_S+|S^\text{ns}|+1)
\end{eqnarray*}
Hence 
\begin{equation*}
Z_{\underline{d}}=\sum_{d_j\leq-2}-(d_j+1)\psi_j
-\sum_{\tiny{\begin{array}{cc}|S^\text{s}|=0\\ d_S+|S^-|<0 \end{array}}}-(d_S+|S^-|)\delta_{0:S}.
\end{equation*}
\end{proof}

%%%%%%%%%%%%%%%%%%%%%%%%%%%%%
%%%%%%%%%%%%%%%%%%%%%%%%%%%%%
%%%%%%%%%%%%%%%%%%%%%%%%%%%%%
\section{Positivity}

We restrict to the case of differentials of the second kind, that is, differentials where all residues are zero. The strata of canonical divisors of the second kind are $\ZZZ(\kappa)$ for $-1\nin\kappa$.
%%%%%%%%%%%%%%%%%%%%%%%%%%%%%
% Compact type Theorem
\Compacttype*

\begin{proof}
Fix $[X,p_1,\dots,p_n]\in\overline{\ZZZ}(\kappa)$. Consider any twisted differential $(X,\{\eta_{X_i}\})$ and level graph $\Gamma$ that shows $[X,p_1,\dots,p_n]$ to be smoothable. By Theorem~\ref{ZRGRC}, the residues at the $p_i$ in the associated $\eta_j$ are zero. Hence the remaining residues at the nodes in every $\eta_j$ sum to zero. 

Every isolated vertex in the level graph gives the condition that the associated residue is equal to zero. Ascending to an isolated vertex sets the residue to zero by Theorem~\ref{ZRGRC}. Descending to an isolated vertex indicates a pole in a differential $\eta_{X_i}$ in which all other residues have been set to zero. Hence the remaining residue must be zero. Further, an isolated horizontal node is not possible as it would indicate an isolated simple pole in a differential $\eta_{X_i}$ in which all other residues have been set to zero. Hence we remove any isolated vertex from a level graph and set the corresponding residue to zero.

Observe $\Gamma$ and any graph obtained from $\Gamma$ by successively removing isolated vertices is a tree graph. However, any tree graph has degree $2(v-1)$ where $v$ is the number of vertices in the graph and hence must contain a vertex of degree one. 
\end{proof}

Boundary strata are idefined as the closure of the locus of curves of a fixed topological type in $\overline{\mathcal{M}}_{g,n}$. Boundary strata are irreducible with codimension given by the number of nodes appearing in the topological type of the general curve. Restricting to genus $g=0$ we obtain the following theorem.
%%%%%%%%%%%%%%%%%%%%%%%%%%%%%
% Restriction lemma in g=0
\Restrictionlemma*

\begin{proof}
If $B$ has codimension $c$ in $\overline{\mathcal{M}}_{0,n}$ then we have introduced $c$ nodes and the generic pointed curve in this boundary strata has $c+1$ components. Hence Theorem~\ref{CT} implies the set of differentials $\{\eta_{X_i}\}$ with signatures $\{\kappa_i\}$ have cumulatively $c$ more poles than $\kappa$. Restricting to each of the $c+1$ components in $B$ we obtain that having a differential of the second kind with signature $\kappa_i$ imposes a codimension $\text{poles}(\kappa_i)-1$ condition. Hence
\begin{eqnarray*}
\text{codim}_B(\overline{\ZZZ}(\kappa)\cap B)= \sum_{i=1}^{c+1}(\text{poles}(\kappa_i)-1)
=  \sum_{i=1}^{c+1}\text{poles}(\kappa_i)  -(c+1)
=\text{poles}(\kappa)-1
\end{eqnarray*}
or the intersection is empty.
\end{proof}

Boundary strata of dimension one are known as $F$-curves. It is conjectured that a divisor $D$ with non-negative intersection with all $F$-curves (such $D$ are said to be $F$-nef) implies non-negative intersection with all effective curves. Known as the $F$-conjecture, this remains open for $\Mgnb$ with $g+n\geq 8$. Gibney, Keel and Morrison~\cite{GKM} showed that the case $\overline{\mathcal{M}}_{0,g+n}$ implies $\Mgnb$. Restricting the strata of differentials of the second kind to genus $g=0$ we obtain.

%%%%%%%%%%%%%%%%%%%%%%%%%%%%%
% F-nef divisors in M0n
\Fnef*

\begin{proof}
Fix an $F$-curve $B$ (an irreducible component of the locus in $\overline{\mathcal{M}}_{0,n}$ with at least $n-4$ nodes) and consider the preimage under $\varphi: \overline{\mathcal{M}}_{0,n+(p-2)}\longrightarrow \overline{\mathcal{M}}_{0,n}$ that forgets the last $p-2$ points where $p=$poles$(\kappa)$. Each irreducible component $\widetilde{B}_i$ for $i=1,\dots, (n-3)^{p-2}$ is a boundary strata of $\overline{\mathcal{M}}_{0,n+(p-2)}$ of dimension $p-1$.

Lemma~\ref{rest} implies dim$(\overline{\ZZZ}(\kappa)\cap \widetilde{B}_i)=0$. Hence as $Z_\kappa^n=\varphi_*\overline{\ZZZ}(\kappa)$, the irreducible curve $B$ is not contained in $Z_\kappa^n$ which implies $B\cdot Z_\kappa^n\geq 0$.
\end{proof}

This theorem covers a large class of divisors. In the cases that we have computed the class of the divisors, as discussed, the classes correspond to classes shown by Fedorchuk~\cite{Fed} to be nef. However we include the inductive argument that the divisors are nef as it may be instructive for the remaining open cases where the classes are unkown. Setting $p=$poles$(\kappa)$ we obtain the following.

%%%%%%%%%%%%%%%%%%%%%%%%%%%%%
% Nef divisors in M0n fo (d,1^{p-2})
\Nef*
\begin{proof}
% Restriction to boundary divisors
Fix a boundary divisor $\delta_{0:S}$. Consider the restriction of $Z_\kappa^n$ to $\delta_{0:S}$. Recall $Z_\kappa^n$ is the cycle $\frac{1}{(p-2)!}\varphi_*\overline{\ZZZ}(\kappa)$ where
\begin{equation*}
\varphi:\overline{\mathcal{M}}_{0,n+(p-2)}\longrightarrow \overline{\mathcal{M}}_{0,n}
\end{equation*}
forgets the last $p-2$ points. The preimage of $\delta_{0:S}$ under $\varphi$ has $2^{p-2}$ irreducible components given by $\delta_{0:S\cup T}$ for $T\subset \{n+1,\dots,n+(p-2)\}$. 

We have
\begin{equation*}
\delta_{0:S\cup T}\cong \overline{\mathcal{M}}_{0,s+t+1}\times  \overline{\mathcal{M}}_{0,n+p-(s+t+1)}
\end{equation*}
where $|S|=s$ and $|T|=t$ and we let $\pi_1$ and $\pi_2$ be the projections to the first and second components respectively. Then 
\begin{equation*}
\delta_{0:S\cup T}\cap \overline{\ZZZ}(\kappa)=\overline{\ZZZ}(\kappa')\times\overline{\ZZZ}(\kappa'')
\end{equation*}
for
\begin{eqnarray*}
\kappa'&=&(n_{S,t},\underline{d}(S),1^t)\\
\kappa''&=&(-2-n_{S,t},\underline{d}(S^c),1^{p-2-t})
\end{eqnarray*}
where $\underline{d}(S)$ is the truncation of the vector $\underline{d}=(d_1,\dots,d_n)$ to include only indices from $S$ and 
\begin{equation*}
n_{S,t}=-2-t-\sum_{i\in S}d_i
\end{equation*}
By Theorem~\ref{existence} this intersection is empty for $n_{S,t}=-1$. Now by symmetry we need only consider $n_{S,t}\geq 0$. In this case $\overline{\ZZZ}(\underline{d}(S),1^t,n_{S,t})$ has codimension $s^--1$ where $s^-=\{i\in S|d_i<0\}$ while $\overline{\ZZZ}(\underline{d}(S^c),1^{p-2-t},-2-n_{S,t})$ has codimension $p-s^--1$. Hence 
\begin{equation*}
\varphi_*[\delta_{0:S\cup T}\cap \overline{\ZZZ}(\kappa)]=0
\end{equation*}
due to each irreducible component dropping dimension when $t\ne s^--1$ or $s^--2$. Hence
\begin{equation}\label{restriction}
Z_\kappa^n\big|_{\delta_{0:S}}=\pi_1^*{Z^{s+1}_{\kappa'}} +  \pi_2^*Z^{n-s+1}_{\kappa''},
\end{equation}
for
\begin{eqnarray*}
\kappa'&=&(n_{S,s^--2},\underline{d}(S),1^{s^--2})\\
\kappa''&=&(-n_{S,s^--1}-2,\underline{d}(S^c),1^{s^--1}).
\end{eqnarray*}
We now proceed by induction. The Theorem is true by Theorem~\ref{Fnef} for $n\leq 7$ as the $F$-conjecture has been proven in this range. Assume true for all $\overline{\mathcal{M}}_{0,n}$ for $n\leq k$. Then for any irreducible effective curve $B$ contained in the boundary of $\overline{\mathcal{M}}_{0,k+1}$ we have 
\begin{equation*}
B\cdot Z^{k+1}_{\kappa}\geq 0
\end{equation*}
for $\kappa=(d_1,\dots,d_{k+1},1^{p-2})$ by Equation~(\ref{restriction}) and the assumption for $n\leq k$. Hence we are left to consider irreducible effective curves $B$ in $\overline{\mathcal{M}}_{0,k+1}$ intersecting the interior $\mathcal{M}_{0,k+1}$. Such $B$ have non-negative intersection with every boundary divisor and hence
\begin{equation*}
B\cdot Z^{k+1}_{\kappa}\geq 0
\end{equation*}
by Theorem~\ref{thm:ZRdivI} as $Z^{k+1}_{\kappa}$ is boundary effective.
\end{proof}

\begin{rem}

\end{rem}

%%%%%%%%%%%%%%%%%%%%%%%%%%%%%%%%%%%%%%%
%%%%%%%%%%%%%%%%%%%%%%%%%%%%%%%%%%%%%%%

%%%%%%%%%%%%%%%%%%%%%%%%%%%%%%%%%%%%%%%
%%%%%%%%%%%%%%%%%%%%%%%%%%%%%%%%%%%%%%%
% Hurwitz Spaces
\section{Irreducibility of certain Hurwitz Spaces}
Consider a degree $d$ branched cover $f:X\longrightarrow Y\cong \PP^1$, branched over $y_1,\dots,y_s$. Choose a point $y\in Y$ outside of this branch locus and label the preimages of $y$ under $f$ as $1,\dots, d$. The \emph{monodromy representation} of the cover $f$ is the group homomorphism
\begin{eqnarray*}\begin{array}{lclll}
\varphi_{f}:&\pi_1(Y\setminus\{y_1,\dots,y_s\},y)&\longrightarrow&S_d\\
&\gamma&\mapsto&\sigma_{\gamma}
\end{array}
\end{eqnarray*}
where $\sigma_{\gamma}$ specifies how the lift of $\gamma$ to $X\setminus\{f^{-1}(y_i)\}_{i=1}^s$ permutes the labelled preimages of $y$. The monodromy representation $\varphi_f$ is well-defined up to conjugation (the choice of labelling the sheets).

Conversely, a group homomorphism
$$\varphi: \pi_1(Y\setminus\{y_1,\dots,y_s\},y)\longrightarrow S_d$$
recovers the cover $f:X\longrightarrow Y$ by the Riemann existence theorem. To restrict to connected degree $d$ covers, we assume that in addition the image of the monodromy representation acts transitively on $\{1,\dots,d\}$.

Consider a basis $\{\gamma_1,\dots,\gamma_s\}$ for $ \pi_1(Y\setminus\{y_1,\dots,y_s\},y)$ consisting of closed loops $\gamma_i$ enclosing only the puncture $y_i$ where $y_1,\dots,y_s$ are in general position. Let 
$$\underline{\eta}=[\eta_1,\dots,\eta_s]$$
be an unordered set of partitions of the integer $d$. The number of conjugacy classes of monodromy representations of $\varphi$ with $\varphi(\gamma_i)$ of cycle type $\eta_i$ is known as \emph{the Hurwitz number of} $\underline{\eta}$, denoted $H_{\underline{\eta}}$, gives the number of distinct covers $f:X\longrightarrow Y\cong\PP^1$ with fixed branch points with ramification profile $\underline{\eta}$.

We define the Hurwitz space of type $\underline{\eta}$ as the space of all covers with a specified branch profile.
$$   \Hur(\underline{\eta}):=\{\hspace{0.2cm}f:C\longrightarrow \PP^1\hspace{0.2cm}\big|\hspace{0.2cm}\text{The ramification of $f$ is of type $\underline{\eta}$}\}.$$                                
The Hurwitz number $H_{\underline{\eta}}$ provides an upper bound for the number of connected components of $\Hur(\underline{\eta})$, however, covers with the same branch points giving different conjugacy classes of monodromy representations can be connected by paths in $\Hur(\underline{\eta})$ obtained by moving the branch points. In this way, the question of the irreducibility of Hurwitz spaces can be viewed as a purely group theoretic question. For example, from this perspective Hurwitz observed that a group theoretic result of Clebsch~\cite{Clebsch} shows all conjugacy classes of monodromy representations of degree $d$ covers of $\PP^1$ where $\underline{\eta}$ contains only transpositions are equivalent under the action of the braid group that moves the branch points. Hence these spaces are irreducible.

Further, as the map $\Hur(\underline{\eta})\longrightarrow \Mg$ forgetting everything but the source curve is dominant in these cases for $d\geq g+1$, Klein~\cite{Klein} noted that this provided a proof that $\Mg$ is irreducible.%~\cite{Hurwitz}.

Kluitmann~\cite{Kluitmann} and Natanzon~\cite{Natanzon} used similar methods to show that the Hurwitz spaces with arbitrary genus source curve and simple branching at all but one branch point are connected. Liu and Osserman~\cite{LO} used limit linear series to show the irreducibility of Hurwitz spaces with genus $g=0$ source curve and pure ramification (the preimage of each branch point contains a unique ramified point, i.e. each partition $\eta_i$ contains a unique value not equal to one). We extend these results and consider cases where ramification is pure at all but one branch point and the source curve is of arbitrary genus. Our method relies on the identification of these type of Hurwitz spaces with subvarieties of the strata of meromorphic differentials.

We define the \emph{stratum of exact abelian differentials with signature $\kappa=(k_1,...,k_m)$} as
\begin{equation*}
\HH_X(\kappa):=\{[C,\omega]\in \HH(\kappa)   \hspace{0.15cm}| \hspace{0.15cm}\int_\gamma\omega=0 \text{ for any closed path $\gamma$ in $C$}\},
\end{equation*}
and the \emph{stratum of exact canonical divisors with signature $\kappa=(k_1,...,k_m)$} as
\begin{equation*}
\X(\kappa):=\{[C,p_1,...,p_m]\in {\mathcal{M}}_{g,m}   \hspace{0.15cm}| \hspace{0.15cm}k_1p_1+...+k_mp_m\sim(\omega)\sim K_C \text{ and $[C,\omega]\in \HH_X(\kappa)$}\}.
\end{equation*}
Again, $\X(\kappa)$ provides a finite cover of $\PP\HH_X(\kappa)$ due to the labelling of the zeros and poles. First we observe the following identification.

\begin{lem}
If $\kappa=(-p_1,\dots,-p_s,a_1,\dots,a_n)$ with $p_i\geq2$, $a_i\geq1$ and either at least two $p_i\geq 3$ or some $p_j\geq 3$ with $p_j\ne a_i+2$ for any $i$, then
\begin{equation*}
\PP\HH_X(\kappa)\cong \Hur(\underline{\eta}_\kappa)
\end{equation*}
where
\begin{equation*}
\underline{\eta}_\kappa=((p_1-1)\cdots(p_s-1),(a_1+1),\dots,(a_n+1)).
\end{equation*}
\end{lem}

\begin{proof}
Given a cover $f:C\longrightarrow \PP^1$ of the specified type, the differential on $C$ is obtained by pulling back a differential on $\PP^1$ with a unique double pole at the first branch point. Given $[C,\omega]\in\HH_X(\kappa)$, the map $f$ is obtained by integration.
\end{proof}

\begin{rem}Observe that the extra requirements on the $p_i$ are so that the special fibre can be identified globally. In general, if some $p_i\geq 3$, then there is a finite cover 
$$\PP\HH_X(\kappa)\longrightarrow \Hur(\eta_{\kappa})$$
due to the possible choice in where to place the double pole on the target curve before pulling back. Further, this identification is always valid locally and hence the degree of $f$ places a restriction on the maximum order of ramification possible. This restriction shows $\X(\kappa)$ and $\HH_X(\kappa)$ are empty if any 
$$ a_i\geq \sum_{i=1}^s(p_i-1).$$
In genus $g=0$ where $\HH_X(\kappa)=\HH_Z(\kappa)$ this is commonly referred to as the "inconvenient vertex" theorem named for the ideas of the flat geometric proof~\cite{GendronTahar}. 
\end{rem}

In genus $g=0$ where $\HH_X(\kappa)=\HH_Z(\kappa)$ this identification and the class given in Theorem~\ref{thm:ZRdivI} allows us to give a new geometric description of the divisor classes introduced by Fedorchuk~\cite[Equation (1.2.1)]{Fed} with $G=\ZZ$ and $f(z)=|z/2|$. We record this as the following Corollary.
\begin{restatable}{cor}{DivCor}\label{DivCor}
Let $\underline{f}=(f_1,\dots,f_n)$ for $f_i\in \ZZ$ and $\sum_{i=1}^n f_i=0$. Let $D_{\underline{f}}$ be the closure of the locus of pointed rational curves $[X,p_1,\dots,p_n]\in\mathcal{M}_{0,n}$ such that there exists a map $f:X\longrightarrow \PP^1$ with ramification order $f_i$ at $p_i$ if $f_i>0$ and one fibre consisting of the $p_i$ such that $f_i<0$ where the ramification at such $p_i$ is $|f_i|-1$. Then the class of $D_{\underline{f}}$ is
$$[D_{\underline{f}}]=\sum_{i=1}^nf(i)\psi_i-\sum_{1\in S}f(S)\delta_{0:S}$$
in $\Pic_\QQ(\overline{\mathcal{M}}_{0,n})$, where $f$ is the symmetric function
$$f(S):=\frac{1}{2}\big| \sum_{i\in S}f_i \big|$$
for $S\subset \{1,\dots,n\}$.
\end{restatable}

Zariski asked if Hurwitz spaces of genus $g$, degree $d$ covers of $\PP^1$ with specified branching over at least $3g$ points were dominant over $\Mg$. This does not hold in general. Liu and Osserman~\cite{LO} asked an altered question about purely ramified Hurwitz spaces that we phrase in our notation.

\begin{question}
Fix $\kappa=(a_1,\dots,a_n,-p,-2^{d-(p-1)},1^{3g})$ with $p\geq 3$ and $a_i\geq 1$ and let
$$\phi:\X(\kappa)\longrightarrow  \mathcal{M}_{g,n+1}$$
be the morphism forgetting all but the first $n+1$ points. Is $\phi$ dominant on every component of $\X(\kappa)$?
\end{question}

Liu and Osserman used limit linear series to show that a positive answer to this question in all cases implied the irreducibility of Hurwitz spaces with arbitrary genus $g$ source curve and pure ramification if there were at least $3g$ simple branch points. We take no position on this question, but related questions arise as we investigate the case with pure branching at all but one branch point.

\begin{lem}\label{dom}
Let $\kappa=(-p_1,\dots,-p_s,a_1,\dots,a_n, 1^{3g+s-1})$ with $p_i\geq 2$ and $a_i\geq 1$, then 
$$\phi:\X(\kappa)\longrightarrow \mathcal{M}_{g,s+n}$$
that forgets all but the first $s+n$ points is dominant.
\end{lem}

\begin{proof}
We proceed by induction. Consider the genus $g=0$ case where 
$$\kappa=(-p_1,\dots,-p_s,a_1,\dots,a_n, 1^{s-1}).$$
Let 
$$\phi':\overline{\mathcal{M}}_{0,2s+n-1}\longrightarrow \overline{\mathcal{M}}_{0,s+n+1}$$
be the morphism forgetting the final $s-2$ points. We know $\phi'_*\overline{\X}(\kappa)$ is codimension one in $\overline{\mathcal{M}}_{0,s+n+1}$ as the restriction to the boundary of strata in higher genus discussed in Section~\ref{class}. If $\X(\kappa)$ is contracted by $\phi$ then the divisor $D^{s+n+1}_\kappa$ must be contracted by 
$$\pi:\overline{\mathcal{M}}_{0,s+n+1}\longrightarrow \overline{\mathcal{M}}_{0,s+n}$$
that forgets the final point. But Theorem~\ref{thm:ZRdivI} gives the class of this divisor. Consider the curve $B$ in $\overline{\mathcal{M}}_{0,s+n+1}$ constructed by fixing the first $s+n$ points in general position on a rational curve and allowing the final point to move freely on the curve. We have (see for example \cite{HarrisMorrison})
$$ B\cdot \psi_i=\begin{cases}n+s-2&\text{for $i=s+n+1$}\\1&\text{for $i\ne s+n+1$}  \end{cases} \hspace{0,7cm}B\cdot \delta_{j,s+n+1}=1$$ 
with all other intersections equal to zero. As $B$ is equal to a general fibre of $\pi$, any divisor that is contracted by $\pi$ must have zero intersection with $B$. But
$$B\cdot [D^{s+n+1}_\kappa]=2s-2$$
by the class formula in Theorem~\ref{thm:ZRdivI}. Hence $\phi$ is dominant in genus $g=0$.

Now assume the result holds in genus $g-1$. In particular, assume the morphism
$$\phi'':\X(-p_1,\dots,-p_s,a_1,\dots,a_n,1^{3(g-1)+s})\longrightarrow \mathcal{M}_{g-1,s+n+1}$$
forgetting the last $3(g-1)+s-1$ points is dominant. Consider the locus
$$W:=\overline{\X}(-p_1,\dots,-p_s,a_1,\dots,a_n,1^{3(g-1)+s})\cup \overline{\X}(-3,1^3)$$
obtained by gluing the $(n+s+1)$th point in the first stratum to the first point in the second stratum. The theory of admissible covers shows this loci is contained $\overline{\X}(\kappa)$. However ${\X}(-3,1^3)$ is simply the locus of elliptic curves with the four two torsion points marked. Hence 
$$\overline{\X}(-3,1^3)\longrightarrow \overline{\mathcal{M}}_{1,1}$$
is dominant and hence $W$ dominates $\delta_{1:\emptyset}$ the locus of stable curves with an unmarked elliptic tail. Hence if $\phi''$ is dominant then $\phi$ is dominant and induction provides the result.
\end{proof}

\begin{lem}\label{onecomp}
Let $\kappa=(-p_1,\dots,-p_s,a_1,\dots,a_n, 1^{3g+s-1})$ with $p_i\geq 2$ and $a_i\geq 1$, then exactly one component of $\X(\kappa)/S_{3g+s-1}$ where $S_{3g+s-1}$ permutes the markings on the final $3g+s-1$ points is dominant under
$$\phi:\X(\kappa)/S_{3g+s-1}\longrightarrow \mathcal{M}_{g,s+n}$$
that forgets all but the first $s+n$ points and $\phi$ has generic degree one.
\end{lem}

\begin{proof}
Consider two dominant components. Their images are both open dense in $\mathcal{M}_{g,s+n}$ and hence intersect on an open dense. Then for every $[C,q_1,\dots,q_{n+s}]\in\mathcal{M}_{g,s+n}$ in this open dense there exist differentials $\omega_1$ and $\omega_2$ with common poles of order $p_i$ at $q_i$ for $i=1,\dots,s$ and common zeros of order $a_i$ at $q_{s+i}$ for $i=1,\dots n$. But then $\alpha\omega_1+\beta\omega_2$ for $[\alpha:\beta]\in \PP^1$ provides a one dimensional fibre of differentials above each point in an open dense of $\mathcal{M}_{g,s+n}$ contradicting the dimension of $\X(\kappa)$
\end{proof}

At this point we introduce some technical results that will be required in our later proofs of irreducibility.

\begin{lem}\label{nonpropgenuszero}
Let
$$\underline{d}=(-p_1,\dots,-p_s,a_1,\dots,a_m)$$
for $\sum_{i=1}^sp_i-\sum_{i=1}^ma_i=s-1$ and $n=s+m\geq 5$, $s,m\geq 2$, $p_i\geq 2$ and $a_i\geq 1$ and 
$$\underline{d}^i=\begin{cases} (-p_1,\dots,-p_i+1,\dots,-p_s,a_1,\dots,a_m)& \text{for $p_i\geq 3$}\\
(-p_1,\dots,-p_{i-1},0,-p_{i+1},\dots,-p_s,a_1,\dots,a_m)  &\text{for $p_i=2$}\end{cases}$$
for $i=1,\dots s$. The classes of divisors $Z^n_{\underline{d}^i}$ as specified in Theorem~\ref{thm:ZRdivI} are not all proportional.
\end{lem}

\begin{proof}
Let 
$$\Theta:\overline{\mathcal{M}}_{0,n-1}\cong\overline{\mathcal{M}}_{0,\{r\}\cup\left\{\{1,\dots,n\}\setminus\{i,j\}\right\}}\longrightarrow\overline{\mathcal{M}}_{0,n}$$
be the boundary map that associates to any $\{r\}\cup\left\{\{1,\dots,n\}\setminus\{i,j\}\right\}$-pointed stable curve, the $n$-pointed stable curve obtained by gluing an $\{i,j,s\}$-pointed rational tail by identifying the point $r$ with $s$. Geometrically or by the pullback formula in section~\ref{maps} and the divisor class given in Theorem~\ref{thm:ZRdivI} we observe
$$ \Theta^*[D^n_{\underline{d}^i}]= \Theta^*[D^n_{\underline{d}^j}]= [D^{n-1}_{\underline{d}^{i,j}}]\ne0$$
for $1\leq i\leq j\leq s$ where
$$ \underline{d}^{i,j}=(-p_i-p_j+2,-p_1,\dots,\hat{-p_i},\dots,\hat{-p_j}\dots,-p_s,a_1,\dots,a_m).$$
Hence if two such divisors are proportional, they are equal.

Next consider the test curve $B_i$ in $\overline{\mathcal{M}}_{0,n}$ constructed by fixing the first $n-1$ points in general position on a rational curve and allowing the final point to move freely on the curve. We have
$$ B_i\cdot \psi_j=\begin{cases}n-3&\text{for $i=j$}\\1&\text{for $i\ne j$}  \end{cases} \hspace{0,7cm}B_i\cdot \delta_{i,j}=1$$ 
with all other intersections equal to zero. Then by the formula given for the class of $[D^n_{\underline{d}^j}]$ we obtain
$$2 B_i\cdot[D^n_{\underline{d}^j}]=\begin{cases}0&\text{for $p_j=2$ and $i=j$}\\
  (m-2)(p_i-1)+\sum_{i=1}^ma_i-\sum_{k=1}^m|a_k-p_i+1|     &\text{for $p_j\geq 2$  and $i\ne j$}\\
 (m-2)(p_i-2)+\sum_{i=1}^ma_i-\sum_{k=1}^m|a_k-p_i+2|   &\text{for $p_j\geq3$ and $i=j$.}   \end{cases}$$
Hence 
$$ 2B_i\cdot\left([D^n_{\underline{d}^i}]-[D^n_{\underline{d}^j}]\right)=2-m+\sum_{k=1}^m\left(|a_k-p_i|-|a_k-p_i+2|\right) $$
for $1\leq i<j\leq s$. This intersection is equal to zero if and only if $p_i\geq a_k+2$ for all but one $k$. Further this must hold for all choices of $i,j$. Without loss of generality we let 
$a_1+1\geq p_i$ for all $i$ and $a_k\leq p_i-2$ for all $i$ and $2\leq k\leq m$.

Finally, consider the curve $\underline{B}_{\{p,q\}}$ obtained by fixing $n-2$ general points on a rational curve $X$ labelled as $\{1,\dots,n\}\setminus\{i,s+1\}$ and gluing an $\{i,s+1,r\}$-pointed rational curve by identifying $r$ with a point that moves freely in $X$. Observe
$$\underline{B}_{\{p,q\}}\cdot \psi_j=\begin{cases}1&\text{for $j\ne p,q$}\\
0&\text{for $j=i,s+1$,}  \end{cases}  
\hspace{0.75cm} \underline{B}_i\cdot\delta_{0:\{i,s+1\}}=n-4,
\hspace{0.75cm} \underline{B}_i\cdot\delta_{0:\{i,j,s+1\}}=1\hspace{0.2cm}\text{for $j\ne i,s+1$.}  $$
Hence for $i=1,\dots,s$
$$2 \underline{B}_{\{i,s+1\}}\cdot[D^n_{\underline{d}^i}]=(n-4)(p_i-a_1-2)+\sum_{k\ne i}(p_k-1)
-\sum_{k\ne i}|a_1-p_i-p_k+3|-\sum_{k=2}^m (a_1-p_i+2)$$
and for $i,j=1,\dots,s$ and $i\ne j$
$$2 \underline{B}_{\{i,s+1\}}\cdot[D^n_{\underline{d}^j}]=
(n-4)(p_i-a_1-1)+\sum_{k\ne i}(p_k-1)-1
-|a_1-p_j-p_i+3|-\sum_{k\ne i,j}|a_1-p_i-p_k+2|-\sum_{k=2}^m (a_1-p_i+1).  $$
Hence
$$2 \underline{B}_{\{i,s+1\}}\cdot\left([D^n_{\underline{d}^i}]-[D^n_{\underline{d}^j}]\right)=6-2m-s+\sum_{k\ne i,j}\left(|a_1-p_i-p_k+2|-|a_1-p_i-p_k+3  | \right)$$
But this expression has a maximum of $4-2m$ which equals zero if and only if $m=2$. Hence $s\geq 3$ and $p_i+p_j\geq a_1+3$ for all $i,j$.

Finally, for $i,k=1,\dots,s$ and $i\ne k$ 
$$2 \underline{B}_{\{i,k\}}\cdot[D^n_{\underline{d}^i}]=(4-n-s)(p_i+p_k-3)+2(a_1+a_2)$$
and for $j\ne i,k$
$$2 \underline{B}_{\{i,k\}}\cdot[D^n_{\underline{d}^j}]=(4-n-s)(p_i+p_k-2)+2(a_1+a_2).$$
But then
$$2 \underline{B}_{\{i,k\}}\cdot\left([D^n_{\underline{d}^i}]-[D^n_{\underline{d}^j}]\right)=n+s-4>0$$
providing the contradiction.
\end{proof}

\begin{lem}\label{nonproport}
Let
$$\kappa=(-p_1,\dots,-p_s,a_1,\dots,a_m,1^{3g+s-1})$$
for $\sum_{i=1}^sp_i-\sum_{i=1}^ma_i=g+s-1$ and $n=s+m$, $s\geq 2$, $p_i\geq 2$ and $a_i\geq 1$ and 
$$\kappa^i=\begin{cases} (-p_1,\dots,-p_i+1,\dots,-p_s,a_1,\dots,a_m,1^{3g+s-2})& \text{for $p_i\geq 3$}\\
(-p_1,\dots,-p_{i-1},0,-p_{i+1},\dots,-p_s,a_1,\dots,a_m,1^{3g+s-3})  &\text{for $p_i=2$}\end{cases}$$
for $i=1,\dots s$. Let 
$$\phi:\overline{\X}(\kappa_i)\longrightarrow \overline{\mathcal{M}}_{g,s+m}$$
be the morphism forgetting all but the first $s+m$ points. Then 
$\phi_*\overline{\X}(\kappa_i)$ is a divisor for all $i$ and the classes $[\phi_*\overline{\X}(\kappa_i)]$ are not all proportional for $i=1,\dots,s$.
\end{lem}

\begin{proof}
We proceed by induction. The result holds for genus $g=0$ by Lemma~\ref{nonpropgenuszero}. Assume the result holds in genus $g-1$.
For ease of notation we let
$$\kappa^i_{g-1}=\begin{cases} (-p_1,\dots,-p_i+1,\dots,-p_s,a_1,\dots,a_m,1^{3(g-1)+s-1})& \text{for $p_i\geq 3$}\\
(-p_1,\dots,-p_{i-1},0,-p_{i+1},\dots,-p_s,a_1,\dots,a_m,1^{3(g-1)+s-2})  &\text{for $p_i=2$}\end{cases}$$
for $i=1,\dots s$. Consider the locus
$$W:=\overline{\X}(\kappa^i_{g-1})\cup \overline{\X}(-3,1^3)$$
obtained by gluing the $(m+s+1)$th point in the first stratum to the first point in the second stratum. The theory of admissible covers shows this loci is contained $\overline{\X}(\kappa^i)$. However ${\X}(-3,1^3)$ is simply the locus of elliptic curves with the four two torsion points marked. Hence 
$$\overline{\X}(-3,1^3)\longrightarrow \overline{\mathcal{M}}_{1,1}$$
is dominant and hence by the induction hypothesis, the image of $W$ has codimension two and is completely contained in the boundary. This implies  $\phi_*\overline{\X}(\kappa_i)$ is a divisor for all $i$. 

Consider the gluing morphism
$$\pi:\overline{\mathcal{M}}_{g-1,s+m+1}\longrightarrow \overline{\mathcal{M}}_{g,s+m}$$
that glues a fixed general one-pointed elliptic curve to the $(s+m+1)$th point. We observe that 
$$\pi^*[\phi_*\overline{\X}(\kappa^i)]=[\phi'_*\overline{\X}(\kappa_{g-1}^i)]$$
where
$$\phi':\overline{\X}(\kappa_{g-1}^i)\longrightarrow \overline{\mathcal{M}}_{g-1,s+m+1}$$
is the morphism forgetting all but the first $s+m+1$ points. Hence if the result holds in genus $g-1$, it holds in genus $g$.
\end{proof}

\Hurwitz*

\begin{proof}
Without loss of generality let 
$$\kappa=(-p_1,\dots,-p_s,a_1,\dots,a_m,1^{3g+(d-1)})$$
be a signature associated to $\underline{\eta}$. (This signature is not unique due to the ordering and also if all but one $p_i=2$, the choice of special fibre.) Then 
$$\X(\kappa)/S_{3g+(s-1)}\longrightarrow \Hur(\underline{\eta})$$
where $S_{3g+(s-1)}$ permutes the markings on the final $3g+(s-1)$ points is a morphism and showing $\X(\kappa)/S_{3g+(s-1)}$ is irreducible implies the result.

We proceed by induction on $d$, the degree of the covers in $\Hur(\underline{\eta})$. The base case is when $d=g+1$, where all ramification is necessarily simple and the result holds by Clebsch~\cite{Clebsch}.

Consider
$$\phi:\overline{\X(\kappa)}/S_{3g+(s-1)}\longrightarrow \overline{\mathcal{M}}_{g,d+m}$$
that forgets all but the first $d+m$ marked points.  Lemma~\ref{dom} implies $\phi$ is dominant and Lemma~\ref{onecomp} implies that there is exactly one component that is dominant under $\phi$. Hence if there is another component it must have positive dimensional fibres under $\phi$.

Fix a point $[C,p_1,\dots,p_{d+m}]\in \mathcal{M}_{g,d+m}$ that has a positive dimensional fibre. Hence one of the forgotten simple zeros moves freely in $C$ in the fibre and in the closure $\overline{\X}(\kappa)$ collides with each $p_i$.

Define 
$$\kappa^i=\begin{cases} (-p_1,\dots,-p_i+1,\dots,-p_s,a_1,\dots,a_m,1^{3g+d-2})& \text{for $p_i\geq 3$}\\
(-p_1,\dots,-p_{i-1},0,-p_{i+1},\dots,-p_s,a_1,\dots,a_m,1^{3g+d-3})  &\text{for $p_i=2$}\end{cases}$$
for $i=1,\dots s$. By the degeneration of differentials of the second kind we observe that 
$$[C,p_1,\dots,p_{d+m}]\in\phi'_*\overline{\X}(\kappa^i)$$
for $i=1,\dots s$, where
$$\phi':\overline{\X}(\kappa^i)\longrightarrow \overline{\mathcal{M}}_{g,d+m}$$
forgets the last $3g+(s-2)$ marked points if $p_i\geq3$ or $3g-(s-3)$ points if $p_i=2$ for $i=1,\dots,s$.

By the inductive hypothesis $\phi'_*\overline{\X}(\kappa^i)$ are irreducible and by Lemma~\ref{nonproport} the classes $[\phi'_*\overline{\X}(\kappa^i)]$ are not all proportional and hence the supports are not all equal and the intersection is of codimension at least two. 

Hence the fibre of $\pi$ above $[C,p_1,\dots,p_{d+m}]$ must have dimension at least two contradicting the inductive hypothesis that $\overline{\X}(\kappa^i)$ are irreducible.

\end{proof}

\exgenusone*

In the case that the source curve has genus $g=0$ this result can be sharpened using the results of Liu and Osserman~\cite{LO}.
%%%%%%%%%%%%%%%%%%%%%%%%%%%%%%%
%%%%%%%%%%%%%%%%%%%%%%%%%%%%%%%
\Hurwitzgenuszero*

\begin{proof}
Without loss of generality let 
$$\kappa=(-p_1,\dots,-p_s,a_1,\dots,a_m,1^{d-3})$$
be a signature associated to $\underline{\eta}$. (This signature is not unique due to the ordering and also if all but one $p_i=2$, the choice of special fibre.) Then again,
$$\X(\kappa)/S_{s-1}\longrightarrow \Hur(\underline{\eta})$$
where $S_{s-1}$ permutes the markings on the final $s-1$ points is a morphism and showing $\X(\kappa)/S_{s-1}$ is irreducible implies the result.

We proceed by induction on $d$. The base case is when $d=3$, where all ramification is necessarily pure and the result holds by Liu and Osserman~\cite{LO}. Further, Liu and Osserman's result covers the pure ramification cases which in our notation are the cases where $s=d-1$ so we restrict to the case $s\leq d-2$.

Consider
$$\phi:\overline{\X(\kappa)}/S_{s-1}\longrightarrow \overline{\mathcal{M}}_{g,d+m}$$
that forgets all but the first $d+m$ marked points.  Lemma~\ref{dom} implies $\phi$ is dominant and Lemma~\ref{onecomp} implies that there is exactly one component that is dominant under $\phi$. Hence if there is another component it must have positive dimensional fibres under $\phi$.

Fix a point $[C,p_1,\dots,p_{d+m}]\in \mathcal{M}_{0,d+m}$ that has a positive dimensional fibre. Hence one of the forgotten simple zeros moves freely in $C$ in the fibre and in the closure $\overline{\X}(\kappa)$ collides with each $p_i$.

Define 
$$\kappa^i=\begin{cases} (-p_1,\dots,-p_i+1,\dots,-p_s,a_1,\dots,a_m,1^{d-4})& \text{for $p_i\geq 3$}\\
(-p_1,\dots,-p_{i-1},0,-p_{i+1},\dots,-p_s,a_1,\dots,a_m,1^{d-5})  &\text{for $p_i=2$}\end{cases}$$
for $i=1,\dots s$. By the degeneration of differentials of the second kind we observe that 
$$[C,p_1,\dots,p_{d+m}]\in\phi'_*\overline{\X}(\kappa^i)$$
for $i=1,\dots s$, where
$$\phi':\overline{\X}(\kappa^i)\longrightarrow \overline{\mathcal{M}}_{g,d+m}$$
forgets the last $s-2$ marked points if $p_i\geq3$ or $s-3$ points if $p_i=2$ for $i=1,\dots,s$.

By the inductive hypothesis $\phi'_*\overline{\X}(\kappa^i)$ are irreducible for $p_i\geq 3$. Though in this case we have introduced a subtlety in that the inductive hypothesis does not necessarily cover $\phi'_*\overline{\X}(\kappa^i)$  for $p_i=2$. However, Lemma~\ref{nonproport} shows the classes $[\phi'_*\overline{\X}(\kappa^i)]$ are not all proportional and hence the supports are not all equal and if the intersection has codimension one, we must have that the support of $\phi'_*\overline{\X}(\kappa^i)$ for each $p_i\geq 3$ forms the same irreducible connected component of the support of $\phi'_*\overline{\X}(\kappa^i)$ for any $p_i=2$. 

Next consider the test curve $B_i$ in $\overline{\mathcal{M}}_{0,n}$ introduced in the proof of Lemma~\ref{nonpropgenuszero} constructed by fixing the first $n-1$ points in general position on a rational curve and allowing the final point to move freely on the curve. We have
$$ B_i\cdot \psi_j=\begin{cases}n-3&\text{for $i=j$}\\1&\text{for $i\ne j$}  \end{cases} \hspace{0,7cm}B_i\cdot \delta_{i,j}=1$$ 
with all other intersections equal to zero. Irreducible curves with class equal to $[B_i]$ cover an open dense of $\overline{\mathcal{M}}_{0,n}$ and hence this curve is a moving curve, that is, $[B_i]\cdot [D]\geq 0$ for all effective divisors $D$. Observe that if $p_i=2$ then
$$B_i\cdot[D^n_{\kappa^j}]=\begin{cases}0&\text{for $i=j$}\\
  m-1>0  &\text{for $i\ne j$ and $p_j\geq3$.}   \end{cases}$$
Hence $\phi'_*\overline{\X}(\kappa^j)$ for each $p_j\geq 3$ cannot form a component of $\phi'_*\overline{\X}(\kappa^i)$ for each $p_i=2$ and the intersection is of codimension at least two. 

Hence the fibre of $\pi$ above $[C,p_1,\dots,p_{d+m}]$ must have dimension at least two contradicting the inductive hypothesis that $\overline{\X}(\kappa^i)$ for $p_i\geq 3$ are irreducible.
\end{proof}

\exgenuszero*

%%%%%%%%%%%%%%%%%%%%%%%%%%%%%%%%%%%%%%%%
%%%%%%%%%%%%%%%%%%%%%%%%%%%%%%%%%%%%%%%%%

\bibliographystyle{plain}
\bibliography{base}

\begin{thebibliography}{1}

\bibitem[AC]{AC}
E. Arbarello and M. Cornalba, The Picard groups of the moduli space of curves.
\emph{Topology }{\bf 26} 153--171, 1987


\bibitem[BCGGM1]{BCGGM1}
M. Bainbridge, D. Chen, Q. Gendron, S. Grushevsky, and M. M\"oller, Compactification of
strata of abelian differentials,  
\emph{Duke. Math. J. }{\bf 167 }(2018), no. 12, 2347--2416.


\bibitem[BCGGM2]{BCGGM2}
M. Bainbridge, D. Chen, Q. Gendron, S. Grushevsky, and M. M\"oller, The moduli space of multi-scale differentials, arXiv:1910.13492


\bibitem[Bo]{Boissy}
C. Boissy, Connected components of the moduli space of meromorphic differentials,  
\emph{Comm. Math. Helv. }{\bf 90} (2015) no. 2, 255--286. 


\bibitem[Ca]{Calta}
K. Calta, Veech surfaces and complete periodicity in genus two. 
\emph{J. Amer. Math. Soc. }{\bf 17} (2004), 871--908.

\bibitem[CTel]{CastravetTevelevHypertree}
A. Castravet and J. Tevelev. Hypertrees, Projections, and Moduli of Stable Rational
Curves. 
\emph{Crelles Journal, }{\bf  675} (2013), 121--180.



\bibitem[CC]{ChenCoskun}
D. Chen and I. Coskun, Extremal effective divisors on $\overline{\mathcal{M}}_{1,n}$.
\emph{Math. Ann. }{\bf 359} (2014), no. 3--4, 891--908.



\bibitem[Cl]{Clebsch}
A. Clebsch, Zur Theorie der Riemann'schen Fl\"{a}chen, 
\emph{Math. Ann. }{\bf 6} (1873), 216--230.




\bibitem[EMa]{EskinMirzakhani}
A. Eskin and M. Mirzakhani, Invariant and stationary measures for the $\SL(2,\RR)$ action on Moduli space, \emph{Publ. IHES }{\bf 127}, no. 1 (2018), 95--324.

\bibitem[EMM]{EskinMirzakhaniMohammadi}
A. Eskin, M. Mirzakhani, and A. Mohammadi, Isolation, equidistribution, and orbit closures for the $\SL(2,\RR)$ action on Moduli space, 
\emph{Ann. Math. }{\bf 182}, no. 2 (2015), 673--721.



\bibitem[FP]{FP}
G. Farkas and R. Pandharipande, The moduli space of twisted canonical divisors, with an
appendix by F. Janda, R. Pandharipande, A. Pixton, and D. Zvonkine, 
\emph{J. Institute Math. Jussieu, }{ \bf 17} (2018), 615--672.
  



\bibitem[Fed]{Fed}
M. Fedorchuk, Semiampleness criteria for divisors on $\overline{\mathcal{M}}_{0,n}$, arXiv:1407.7839 
 

\bibitem[Fil]{Filip} 
S. Filip, Splitting mixed Hodge structures over affine invariant manifolds, 
 \emph{Ann. of Math. (2) }{\bf 183} (2016), no. 2, 681--713.
 

\bibitem[GT]{GendronTahar} 
Q. Gendron and G. Tahar, Diff\'{e}rentielles \`{a} singularit\'{e}s prescrites, arXiv:1705.03240

\bibitem[GKM]{GKM}
A. Gibney, S. Keel, and I. Morrison, Towards the ample cone of $\Mgnb$, 
\emph{J. Amer. Math. Soc. }{\bf 15(2)} (2002), 273--294.

\bibitem[GHS]{GHS}
T. Graber, J. Harris and J. Starr, A note on Hurwitz schemes of covers of a positive genus
curve, preprint, arXiv:math/0205056 (2002)



\bibitem[GZ]{GruZak}
S Grushevsky and D Zakharov, The double ramification cycle and the theta divisor, 
\emph{Proc. Amer. Math. Soc. }{\bf 142} (2014), no. 12, 4053--4064.

\bibitem[HMo]{HarrisMorrison}
J. Harris and I. Morrison, Moduli of curves, 
\emph{Graduate Texts in Mathematics} {\bf 187}, Springer-Verlag New York, 1998.


 
\bibitem[K]{Keel} 
S. Keel, Intersection theory of moduli space of stable n-pointed curves of genus zero. 
\emph{Trans. Amer. Math. Soc. }{\bf 330} (2), 545--574, 1992.
 


\bibitem[Kl]{Klein} 
F. Klein, \"{U}ber Riemann's Theorie de Algebraischen Functionen. Leipzig: Teubner 1882.

\bibitem[Klu]{Kluitmann}
P. Kluitmann, Hurwitz action and finite quotients of braid groups. 
\emph{Braids (Santa Cruz, CA 1986). contemporary Mathematics, } vol. 78, pp. 299--325. AMS, Providence (1988).

\bibitem[LO]{LO}
F. Liu and B Osserman, The irreducibility of certain pure-cycle Hurwitz spaces. 
\emph{Amer. J. Math. }{\bf 130 }(2008), no. 6, 1687--1708.


\bibitem[Mc]{McMullen}
 C. McMullen, Billiards and Teichm\"uller curves on Hilbert modular surfaces. 
 \emph{J. Amer. Math. Soc. }{\bf 16} (2003), 857--885.

\bibitem[MMW]{McMW}
C. McMullen, R. Mukamel, A. Wright, Cubic curves and totally geodesic subvarieties of moduli space,
\emph{Ann. of Math, }{\bf 185(3)} (2017), 957--990.


\bibitem[M1]{Mullane1}
S. Mullane, Divisorial strata of abelian differentials,
\emph{Int. Math. Res. Notices }{\bf 6 }(2017), 1717--1748.
%
\bibitem[M2]{Mullane2}
S. Mullane, Effective divisors in $\Mgnb$ from abelian differentials, 
\emph{Michigan Math. J. }{\bf 67 }(2018), 839--889.

\bibitem[M3]{Mullane3}
S. Mullane, On the effective cone of $\Mgnb$, 
\emph{Adv. in Math. }{\bf 320 }(2017), 500--519.

\bibitem[M4]{Mullane4}
S. Mullane, On the effective cone of higher codimension cycles in $\Mgnb$, 
\emph{Mathematische Zeitschrift }{\bf 295 }(2020), no. 1, 265--288.


\bibitem[M\"u]{Muller}
 F. M\"uller, The pullback of a theta divisor to $\Mgnb$, 
 \emph{Math. Nachr. }{\bf 286} (2013), no. 11-12,
1255--1266.
%
%


\bibitem[N]{Natanzon}
S.M. Natanzon, Topology of 2-dimensional coverings and meromorphic functions
on real and complex algebraic curves. 
\emph{Selected Translations. Sel. Math. Sov. }{\bf 12(3)} (1993)
251--291.



\bibitem[O]{Opie}
M. Opie, Extremal divisors on moduli spaces of rational curves with marked points. 
\emph{Michigan Math. J. }{\bf 65} (2016), no. 2, 251--285.



\bibitem[P]{Pixton}
A. Pixton, A nonboundary nef divisor on $\overline{\mathcal{M}}_{0,12}$. 
\emph{Geom. Topol. }{\bf 17(3)} (2013), 1317--1324.


\bibitem[V]{Veech}
W. Veech, Gauss measures for transformations on the space of interval exchange maps,
\emph{Ann. of Math. (2), }{\bf 115(1)} (1982), 201--242

\bibitem[Vem]{KV}
P. Vermeire, A counterexample to Fulton's conjecture on $\MOnb$, 
\emph{J. Algebra, }{\bf 248} (2002), no. 2, 780--784.



  \end{thebibliography}
\end{document}